\documentclass[twoside,11pt]{article} 
\usepackage{jmlr2e} 
\usepackage[utf8]{inputenc}

\usepackage{color}
\usepackage[utf8]{inputenc}
\usepackage[T1]{fontenc}
\usepackage{amsfonts}
\usepackage{amsmath}
\usepackage{amssymb}
\usepackage{verbatim}
\usepackage[margin=1.2in]{geometry}
\usepackage{bbm}
\usepackage{enumitem}
\usepackage{hyperref}
\usepackage{mathtools}

\usepackage{natbib}
\usepackage{graphicx}

\newcommand*\mcupinn[2]{\vcenter{\hbox{$\mathsurround=0pt
  \ifx\displaystyle#1\textstyle\else#1\fi\bigcup$}}}

\newcommand*\mcapinn[2]{\vcenter{\hbox{$\mathsurround=0pt
  \ifx\displaystyle#1\textstyle\else#1\fi\bigcap$}}}
\DeclarePairedDelimiter{\abs}{\lvert}{\rvert}
\DeclarePairedDelimiter{\norm}{\lVert}{\rVert}
\DeclarePairedDelimiter{\parr}{(}{)}
\DeclarePairedDelimiter{\parq}{[}{]}

\DeclarePairedDelimiter{\bra}{\lbrace}{\rbrace}

\DeclarePairedDelimiter{\prodscal}{\langle}{\rangle}

\DeclareMathOperator*{\argmin}{arg\,min}

\DeclareMathOperator*{\expval}{\mathbb{E}}

\DeclareMathOperator*{\tr}{tr}
\DeclareMathOperator*{\st}{\,:\,}
\DeclareMathOperator*{\rk}{\mathrm{rk}}
\newcommand{\bfun}{{\mathbf f}}
\newcommand{\bg}{{\mathbf g}}
\newcommand{\bh}{{\mathbf h}}
\newcommand{\bq}{{\mathbf q}}
\newcommand{\bx}{{\mathbf x}}
\newcommand{\by}{{\mathbf y}}
\newcommand{\bw}{{\mathbf w}}
\newcommand{\be}{{\mathbf e}}

\newcommand{\bb}{{\mathbf b}}
\newcommand{\bv}{{\mathbf v}}
\newcommand{\bu}{{\mathbf u}}
\newcommand{\bz}{{\mathbf z}}
\newcommand{\bG}{{\mathbf G}}
\newcommand{\bI}{{\mathbf I}}
\newcommand{\bK}{{\mathbf K}}
\newcommand{\bM}{{\mathbf M}}
\newcommand{\bX}{{\mathbf X}}
\newcommand{\bY}{{\mathbf Y}}
\newcommand{\bT}{{\mathbf T}}
\newcommand{\bP}{{\mathbf P}}
\newcommand{\bA}{{\mathbf A}}
\newcommand{\bO}{{\mathbf O}}
\newcommand{\bU}{{\mathbf U}}
\newcommand{\bV}{{\mathbf V}}
\newcommand{\bQ}{{\mathbf Q}}
\newcommand{\bW}{{\mathbf W}}
\newcommand{\bzero}{{\mathbf 0}}
\newcommand{\bSigma}{{\boldsymbol \Sigma}}
\newcommand{\bLambda}{{\boldsymbol \Lambda}}
\newcommand{\balpha}{{\boldsymbol \alpha}}
\newcommand{\bbeta}{{\boldsymbol \beta}}
\newcommand{\bgamma}{{\boldsymbol \gamma}}
\newcommand{\btheta}{{\boldsymbol \theta}}
\newcommand{\bPhi}{{\boldsymbol \Phi}}
\newcommand{\bvarphi}{{\boldsymbol \varphi}}
\newcommand{\bpsi}{{\boldsymbol \psi}}
\newcommand{\R}{{\mathbb R}}
\newcommand{\E}{{\mathbb E}}

\usepackage{lastpage}
\ShortHeadings{Spurious Valleys in One-hidden-layer Neural Network Optimization Landscapes}{Venturi, Bandeira and Bruna}

\begin{document}

\title{Spurious Valleys in One-hidden-layer Neural Network Optimization Landscapes}

\author{\name Luca Venturi \email venturi@cims.nyu.edu \\
      \addr Courant Institute of Mathematical Sciences, 251 Mercer Street, New York, NY 10012 
      \AND
      \name Afonso S. Bandeira \email bandeira@cims.nyu.edu \\
      \name Joan Bruna \email bruna@cims.nyu.edu \\
      \addr Courant Institute of Mathematical Sciences, 251 Mercer Street, New York, NY 10012 \\
      \addr Center for Data Science, 60 5th Avenue, New York, NY 10011}


\maketitle

\begin{abstract}
Neural networks provide a rich class of high-dimensional, non-convex optimization problems. Despite their non-convexity, gradient-descent methods often successfully optimize these models. This has motivated a recent spur in research attempting to characterize properties of their loss surface that may explain such success. 

In this paper, we address this phenomenon by studying a key topological property of the loss: the presence or absence of \emph{spurious valleys}, defined as connected components of sub-level sets that do not include a global minimum.
Focusing on a class of one-hidden-layer neural networks defined by smooth (but generally non-linear) activation functions, 
we identify a notion of intrinsic dimension and show that it provides necessary and sufficient conditions for the 
absence of spurious valleys. More concretely, finite intrinsic dimension guarantees that for sufficiently overparametrised 
models no spurious valleys exist, independently of the data distribution. Conversely, infinite intrinsic dimension implies
that spurious valleys do exist for certain data distributions, independently of model overparametrisation. 
Besides these positive and negative results, we show that, although spurious valleys may exist in general, they are confined to low risk levels and avoided with high probability on overparametrised models.

\end{abstract}

\section{Introduction}\label{section:introduction}

Modern machine learning applications involve datasets of increasing dimensionality, complexity and size, which in turn motivate the use of high-dimensional, non-linear models, as illustrated in many deep learning algorithms across computer vision, speech and natural language understanding. The prevalent strategy for learning is to rely on Stochastic Gradient Descent (SGD) methods, that typically operate on non-convex objectives. In this context, an outstanding goal is to provide a theoretical framework that explains under what conditions -- relating input data distribution, choice of architecture and choice of optimization scheme -- this setup will be successful. 

More precisely, let $\bPhi_\btheta\, :\, \mathbb{R}^n \to \mathbb{R}^m$  denote a model class parametrized by $\btheta \in \Theta\subseteq  \mathbb{R}^P$, which in the case of Neural Networks (NNs) contains the aggregated weights across all layers. In a supervised learning setting, this model is deployed on some data $(\bX,\bY)$ random variable taking values in $\mathbb{R}^n\times \mathbb{R}^m$, to predict targets $\bY$ given input $\bX$, and its risk for a given $\btheta$ is 
\begin{equation}
\label{riskeq}
L(\btheta) = \mathbb{E}_{(\bX,\bY) \sim P} \parq*{\ell( \bPhi_\btheta(\bX), \bY)}
\end{equation}
where $\ell$ is a convex loss, such as a square loss or a logistic regression loss. In the following we refer to \eqref{riskeq} as the risk, the energy or the loss interchangeably. The aim is to find $ \btheta^* \in \argmin_{\btheta \in \Theta} L(\btheta)$
and this is attempted in practice by running SGD iteration.
Under some technical conditions, the expected gradient is known to converge to zero \citep{bottou2016optimization}. Understanding the nature of such stationary points - and therefore the landscape of the loss function - is a task of fundamental importance to understand performance of SGD. 

Whereas there is a growing literature in analyzing the behavior of SGD on non-convex objectives \citep{soudry2017implicit, ji2018risk, gunasekar2018characterizing, wilson2017marginal}, 
we focus here on properties of the optimization problem above that are algorithm independent. 
A common factor shared in the above cited works (and in common practice) is that \emph{overparametrisation} of the model class (i.e. $P \gg 1$) often leads to improved performance, despite the potential increase in generalization error.



Our analysis focuses mostly on the class of one-hidden-layer neural networks, with a hidden layer of size $p$, and covers both empirical and population risk landscapes. 
More specifically, we look at presence (or absence) of \emph{spurious valleys}, defined as connected components of the sub-level sets that do not contain a global minima. 
We define two quantities depending on the functional space spanned by neural networks of different widths: the upper intrinsic dimension, defined as the dimension of this linear space, and the lower intrinsic dimension, defined as the minimum number of hidden units to describe any element of the functional space.
Upper and lower intrinsic dimensions define only two scenarios: either (i) they are both finite, enabling positive results; or (ii) they are both infinite, implying the negative results.

\paragraph{Summary of contributions}

More specifically, we show that:
\begin{itemize}
\item For Empirical Risk Minimization or polynomial activations, spurious valleys do not occur as long as the network is sufficiently over-parametrised. For the case of linear and quadratic activations, our results are (up to a constant factor) tight.
\item For non-polynomial non-negative activations, for any hidden width, we construct data distributions which yield spurious valleys with positive measure, whose value is arbitrarily far from the one of the global.
\item Finally, drawing on connections with random features expansions, we show that, even if spurious valleys may appear in general, their measure decreases as the width increases. This holds up to a low energy threshold, which approaches the global minimum at a rate inversely proportional to the hidden layer size (up to log factors). 
\end{itemize}

\paragraph{Related works}

A considerable
amount of literature has attempted to characterize the landscape of the loss function \eqref{riskeq} by studying its critical points. Global optimality results have been obtained for NN architectures with  linear activations \citep{hardt2016identity,kawaguchi2016deep,yun2018small}, quadratic activations~\citep{soltanolkotabi2017theoretical,du2018power} and some more general non-linear activations, under appropriate regularity assumptions \citep{soudry2016no,nguyen2017loss,feizi2017porcupine}. Some other insights have been obtained by leveraging tools for complexity analysis of spin glasses \citep{choromanska2015loss} and random matrix theory \citep{pennington2017geometry}.
Other analysis involved studying goodness of the initialization of the parameter values $\btheta_0$~\citep{daniely2016toward,safran2016quality,du2017gradient} or other topological properties of the loss \eqref{riskeq}, such as connectivity of sub-level sets \citep{draxler2018essentially,brunatopology}.

Several other type of analysis of the convergence of NNs gradient-based optimization algorithms have been considered in the literature. For example, \citep{ge2017learning} proved convergence of GD on a modified loss; \citep{shamir2018resnets} compared optimization properties of residual networks with respect to linear models; in  \citep{dauphin2014identifying} it is argued that the issues arising in the optimization of NN architectures are due to the presence of saddle points in the loss function rather than spurious local minima. Optimization landscapes have also been studied in other contexts than from NNs training, such as non-convex low rank problems \citep{ge2017no}, matrix completion \citep{ge2016matrix}, problems arising in semidefinite programming \citep{boumal2016non,bandeira2016low} and implicit generative modeling \citep{bottou2017geometrical}.

\paragraph{Structure of the paper}

The rest of the paper is structured as follows. Section \ref{section:problem_setting} formally introduces the notion
of spurious valleys and explains why this is a relevant concept from the optimization point of view. It also defines the intrinsic dimensions of a network (Section \ref{sec:dimension}).  
In Section~\ref{sec:no spurious valleys} we state our main positive results (Theorem \ref{master_theorem}) and we discuss two settings where they bear fruit: polynomial activation functions and empirical risk minimization.
Section \ref{sec:spurious valleys exist} is dedicated to constructions of worst case scenarios for activation with infinite lower intrinsic dimension. We then show, in Section \ref{sec:gap}, that, even if spurious valleys may exist, they tend to be confined to regimes of low risk.  
Some conclusive discussion is reported in Section \ref{section:discussion}.

\subsection{Notation}\label{sec:notation}

We introduce notation we use throughout the rest of the paper.
For any integers $n\leq m$ we denote $[n,m] = \bra{n,n+1,\dots,m}$ and, if $n >0$, $[n] = [1,n]$.
We denote scalar valued variables as lowercase non-bold; vector valued variables as lowercase bold; matrix and tensor valued variables and multivariate random variables (r.v.'s) as uppercase bold. Given a vector $\bv\in\mathbb{R}^n$, we denote its components as $v_i$; given a matrix $\bW\in\mathbb{R}^{n\times m}$, we denote its rows as $\bw_i$;  given a tensor $\bT\in\mathbb{R}^{n_1\times \cdots \times n_k}$, we denote its components as $T_{i_1\cdots i_k}$. Given some vectors $\bv_i \in \mathbb{R}^{n_i}$, $i\in[k]$, the tensor product $\bv_1\otimes \cdots \otimes \bv_k$ denotes the $n_1\times \cdots \times n_k$ dimensional tensor $\bT$ whose components are given by $T_{i_1\cdots i_k} = v_{i_1}\cdots v_{i_k}$; given a vector $\bv$, we denote $\bv ^{\otimes k} = \otimes_{i=1}^k \bv$. We denote by $\mathrm{S}^{k}\parr*{\mathbb{R}^{n}}$ the space of order $k$ symmetric tensors on $\mathbb{R}^n$. 
For any $\bT \in \mathrm{S}^{k}\parr*{\mathbb{R}^{n}}$, we define the symmetric rank \citep{comon2008symmetric} as 
$\rk_\mathrm{S}(\bT) = \min\bra*{ p \geq 1 \st \bT = \sum_{i=1}^p u_i\bw_i^{\otimes k}\text{ for some }\bu\in\mathbb{R}^p,\bw_1,\dots,\bw_p\in\mathbb{R}^n}$.
We define $\rk_\mathrm{S}(k,n) = \max\bra{ \rk_\mathrm{S}(\bT) \st \bT \in \mathrm{S}^k(\mathbb{R}^n) }$.
Finally, $\mathbb{S}^{n-1}\subset \mathbb{R}^n$ denotes the $(n-1)$-dimensional sphere $\bra*{\bx\in\mathbb{R}^n\st \norm{\bx}=1}$.

\section{Preliminaries}\label{section:problem_setting}

\subsection{Problem setup}

Let $(\bX,\bY)$ be two r.v.'s. These r.v.'s take values in $\mathbb{R}^n$ and $\mathbb{R}^m$ and represent the \emph{input} and \emph{output} data, respectively. 
We consider oracle square loss functions $L:\Theta\to\mathbb{R}$ of the form
\begin{equation}
L(\btheta) \doteq \expval\parq*{\ell\parr*{\bPhi(\bX;\btheta),\bY}} \label{eq:lossFun}
\end{equation}
where $\ell : \mathbb{R}^m\times \mathbb{R}^m \to [0,\infty)$ is convex in its first argument. For every $\btheta\in\Theta$, the function $\bPhi(\cdot;\btheta) : \mathbb{R}^n \to \mathbb{R}^m$ models the dependence of the output on the input as $\bY\simeq \bPhi(\bX;\btheta)$. We focus on one-hidden-layer NN functions $\bPhi$, i.e. $\bPhi$ of the form
\begin{equation}\label{eq:nn_feature_function}
\bPhi( \bx; \btheta) = \bU\sigma(\bW\bx)
\end{equation}
where $\btheta = (\bU,\bW) \in \Theta \doteq \mathbb{R}^{m\times p}\times \mathbb{R}^{p\times n}$. Here $p$ represents the width of the hidden layer and $\sigma:\mathbb{R}\to\mathbb{R}$ is a continuous element-wise \emph{activation} function. 

The loss function $L(\btheta)$ is (in general) a non-convex object; it may present spurious (i.e. non global) local minima. In this work, we characterize $L(\btheta)$ by determining absence or presence of spurious valleys, as defined below.
\begin{definition}
For all $c\in\mathbb{R}$ we define the sub-level set of $L$ as $\Omega_L(c) = \bra{\btheta \in \Theta\st L(\btheta) \leq c}$. We define a \emph{spurious valley} as a path-connected component of a sub-level set $\Omega_L(c)$ which does not contain a global minimum of the loss $L(\btheta)$.
\end{definition}
Since, in practice, the loss \eqref{eq:lossFun} is minimized with a gradient descent based algorithm, then absence of spurious valleys is a desirable property, if we wish the algorithm to converge to an optimal parameter. It is easy to see that $L(\btheta)$ not having spurious valleys is implied by the following property: 

\begin{enumerate}[label=\textbf{P.\arabic*}]
\item \label{p:conn} Given any \emph{initial} parameter $\tilde{\btheta}\in\Theta$, there exists a continuous path
$\btheta: t\in[0,1] \mapsto \btheta_t \in \Theta$ such that: 
\begin{enumerate}
\item $\btheta_0 = \tilde{\btheta}$
\item $\btheta_1 \in \arg\min_{\btheta\in\Theta} L(\btheta) $
\item The function $t\in[0,1]\mapsto L(\btheta_t)$ is non-increasing
\end{enumerate}
\end{enumerate}
As pointed out in \citep{brunatopology}, this implies that $L$ has no strict spurious (i.e. non global) local minima. 
The absence of generic (i.e. non-strict) spurious local minima 
is guaranteed if the path $\btheta_t$ is such that the function $L(\btheta_t)$ is strictly decreasing. 
For sake of clarity, we review these properties in the following lemma (the proof is reported in the Appendix \ref{app:lemmas}).
\begin{lemma}\label{lemma:what_are_spurious_valleys}
Be $\btheta \mapsto L(\btheta)$ a continuous function. Then, property \ref{p:conn} implies absence of spurious valleys. In particular, this implies absence of strict spurious minima, and of (generally non-strict) spurious minima if property \ref{p:conn} holds with strictly decreasing paths $t\mapsto L(\btheta_t)$. Conversely, presence of spurious valleys implies existence of spurious minima.
\end{lemma}

In the following, we prove absence of spurious valleys by proving that property \ref{p:conn} holds. 
Intuitively, we should think about spurious valleys as regions of the parameter space from which it is impossible to `escape' without `up-climbing' the loss value.

Notice that for many activation functions used in practice (such as the ReLU $\sigma(z) = z_+$), the parameter $\btheta$ determining the function $\bPhi(\cdot;\btheta)$ is determined up to the action of a symmetry group (e.g., in the case of the ReLU, $\sigma$ is a positive homogeneous function). This already prevents strict minima: for any value of the parameter $\btheta\in\Theta$ there exists a (often large) \emph{manifold} $\mathcal{U}_\btheta\subset \Theta$ intersecting $\btheta$ along which the loss function is constant.  

\paragraph{ERM vs population loss} In the following, we consider the loss \eqref{eq:lossFun} defined for a generic distribution $(\bX,\bY)$. In case of a distribution with a finite number of atoms, this corresponds to empirical risk minimization (ERM), which is (usually) the regime where machine learning algorithms perform optimization. On the other hand, for a generic data distribution, this loss is what is called \emph{population} loss, and corresponds to the actual objective that machine learning algorithms aim to minimize. In our work we are interested in analyzing not only the ERM case, but more general population losses. While we in fact focus on highly over-parametrised neural networks, we aim to provide results which apply to the regime where number of data points goes to infinity before the number of parameters.

\subsection{Intrinsic dimension of a network}\label{sec:dimension}

The main result of this work is to exploit that the property of absence of spurious valleys is related to the complexity of the functional space $V_\sigma = \{ f = \Phi_\btheta \st \btheta \in \Theta\}$ defined by the network architecture. We therefore define two measures of such complexity which we will use to show, respectively, positive and negative results in this regard. 

To simplify the discussion, we introduce some notation which we will use throughout the rest of the paper. Let $\sigma:\mathbb{R}\to\mathbb{R}$ be a continuous activation function. For every $\bv\in\mathbb{R}^n$ we denote $\psi_{\sigma,\bv}$ to be the function $\psi_{\sigma,\bv}:\bx\in\mathbb{R}^n \mapsto \sigma(\prodscal{\bv,\bx})\in\mathbb{R}$. We refer to each $\psi_{\sigma,\bv}$ as a \emph{filter} function. If $\bX$ is a r.v. taking values in $\mathbb{R}^n$, we denote by $L^2_\bX$ the space of square integrable function on $\mathbb{R}^n$ w.r.t. the probability measure induced by the r.v. $\bX$. We then define the two following functional spaces:
\begin{align*}
V_{\sigma,p} & = \bra*{ f = \Phi(\cdot;\btheta) \st \btheta = (\bu^T,\bW)\in \Theta = \mathbb{R}^p\times \mathbb{R}^{p\times n }} \\
\mathcal{R}_2(\sigma,n) & = \bra*{ \bX \text{ r.v. taking values in  } \mathbb{R}^n \st \psi_{\sigma,\bv}\in L^2_\bX \text{ for every } \bv\in\mathbb{R}^n }
\end{align*}
$V_{\sigma,p}$ represents the space of (one-dimensional output) functions modeled by the network architecture and  $\mathcal{R}_2(\sigma,n)$ to be the space of ($n$-dimensional) input data distributions for which the filter functions have finite second moment. We finally define 
\begin{equation*}
V_\sigma = \mathrm{span}\parr*{ \bra*{ f \st f \in V_{\sigma,1} } } = \bigcup_{p=1}^\infty V_{\sigma,p}
\end{equation*}
as the linear space spanned by the functions $\psi_{\bv,\sigma}$ for $\bv \in \mathbb{R}^n$.
\begin{definition}\label{def:uid}
Let $\sigma$ be a continuous activation function and $\bX\in \mathcal{R}_2(\sigma,n)$ a r.v. We define\footnote{For any linear subspace $V \subseteq L^2_\bX$, $\mathrm{dim}_{L^2_\bX}(V)$ denotes the dimension of $V$ as a subspace of $L^2_\bX$.}
\begin{equation*}
\mathrm{dim}^*(\sigma,\bX) = \mathrm{\dim}_{L^2_\bX}\parr*{V_\sigma}
\end{equation*}
as the upper intrinsic dimension of the pair $(\sigma,\bX)$.
We define the level $n$ upper intrinsic dimension of $\sigma$ as $\mathrm{dim}^*(\sigma,n) = \mathrm{\dim}\parr*{V_\sigma} = \sup\bra{ \mathrm{dim}^*(\sigma,\bX) \st \bX \in \mathcal{R}_2(\sigma,n)}$.
\end{definition}

The upper intrinsic dimension $\dim^*(\sigma,\bX)$ defined above is therefore the dimension of the functional space spanned by the filter functions $\psi_{\sigma,\bv}\in L^2_\bX$ or, equivalently, of the image of the map $\Phi:\btheta \in \Theta \mapsto \Phi(\cdot;\btheta) \in L^2_\bX$. Notice that $\mathrm{dim}^*(\sigma,\bX) \leq \dim(L^2_\bX)$. In particular, if the distribution $\bX$ is discrete, i.e. it is concentrated on a finite number of points $\bra{\bx_1,\dots,\bx_N} \subset \mathbb{R}^n$, then $\mathrm{dim}^*(\sigma,\bX) \leq \dim(L^2_\bX) \leq N$. Otherwise, if the distribution $\bX$ is not discrete, then $\dim(L^2_\bX) = \infty$. 

The $n$ level upper intrinsic dimension $\dim^*(\sigma,n)$ is defined as the dimension of the functional linear space $V_\sigma$. We note that if $\bX\in\mathcal{R}_2(\sigma,n)$ is a r.v. with almost surely (a.s.) positive density w.r.t. the Lebesgue measure $dx$, then $\dim^*(\sigma,n) = \dim^*(\sigma,\bX)$.

The following lemma exhausts all the cases when the upper intrinsic dimension is not infinite.

\begin{lemma}\label{lemma:uid_when_infty}
Let $\sigma$ be a continuous activation function and $\bX \in \mathcal{R}_2(\sigma,n)$ such that $\dim\parr*{L^2_\bX} = \infty$. If $\sigma(z) = \sum_{k=0}^d a_k z^k$ is a polynomial, then
\begin{equation*}
\dim^*(\sigma,\bX) \leq \sum_{i=1}^d\binom{n + i -1}{i}\mathbf{1}_{\bra{a_i\neq 0}} = O(n^d)
\end{equation*}
Otherwise (i.e. if $\sigma$ is not a polynomial) it holds $\dim^*(\sigma,\bX) = \infty$. 
\end{lemma}

The proof of the above lemma is based on the universal approximation theorem \citep{leshno1993multilayer}. We then define the lower intrinsic dimension, which corresponds to the concept of `how many hidden neurons are needed to represent a generic function of $V_\sigma$'.

\begin{definition}\label{def:lid}
Let $\sigma$ be a continuous activation function and $\bX\in \mathcal{R}_2(\sigma,n)$ a r.v. We define\footnote{For any subsets $V,W \subseteq L^2_\bX$, we say that $V \subsetneq_{L^2_\bX} W$ if $V \subsetneq W$ as subsets of $L^2_\bX$ (and similar with other inclusions or equalities).}
\begin{equation*}
\mathrm{dim}_*(\sigma,\bX) = \max\bra*{p \geq 1 ~\st~ V_{\sigma,p-1} \subsetneq_{L^2_\bX} V_{\sigma,p}}
\end{equation*}
as the lower dimension of the pair $(\sigma,\bX)$. We define the level $n$ lower dimension of $\sigma$ as $\mathrm{dim}_*(\sigma,n) = \max\bra*{p \geq 1 ~\st~ V_{\sigma,p-1} \subsetneq V_{\sigma,p}} = \sup\bra*{\mathrm{dim}_*(\sigma,\bX) \st \bX\in \mathcal{R}_2(\sigma,n)}$.
\end{definition}

If $\dim_*(\sigma,\bX)$ is finite, then it corresponds to the minimum number of hidden neurons which are needed to represent any function of $V_\sigma$ with the NN architecture \eqref{eq:nn_feature_function}. Clearly, this implies that 
\begin{equation*}
\dim_*(\sigma,\bX) \leq \dim^*(\sigma,\bX)
\end{equation*}
for every continuous activation function $\sigma$ and any $\bX \in \mathcal{R}_2(\sigma,n)$. As with the upper instrinsic dimension, we note that if $\bX\in\mathcal{R}_2(\sigma,n)$ is a r.v. with a.s. positive density w.r.t. the Lebesgue measure $dx$, then $\dim_*(\sigma,n) = \dim_*(\sigma,\bX)$.

In the case of homogeneus polynomial activations $\sigma(z) = z^k$ with $k\geq 1$ integer, the level $n$ lower dimension of $\sigma$ coincides with the notion of (maximal) symmetric tensor rank. 

\begin{lemma}\label{lemma:lid}
Let $\sigma(z) = z^k$, with $k$ positive integer. Then 
\begin{equation*}
\mathrm{dim}_*(\sigma,n) = \mathrm{rk}_\mathrm{S}(k,n)
\end{equation*}
\end{lemma}
Finally, the next lemma implies that for most non-polynomial activation functions practical interest, the lower intrinsic dimension $\dim_*(\sigma,n)$ is infinite.

\begin{lemma}\label{lemma:infinite lower dimension}
Let $\sigma$ be a continuous activation function such that $\sigma \in L^2(\mathbb{R}, e^{-x^2/2}\,dx)$ and $n>1$. Then $\dim_*(\sigma,n) = \infty$ if and only if $\sigma$ is not a polynomial.
\end{lemma}

The proof of the above Lemma is based on Hermite decomposition and on the correspondence between one-hidden-layer nets and symmetric tensors \citep{mondelli2018connection}.

\section{Finite intrinsic dimension and absence of spurious valleys}\label{sec:no spurious valleys}

In this section we provide our positive results. Essentially they state that if the width of the network matches the dimension of the functional space $V_\sigma$ spanned by its filter functions, then no spurious valleys exist. We first provide the main result (Theorem \ref{master_theorem}) in a general form, which allows a straight-forward derivation of two cases of interest: empirical risk minimization (Corollary \ref{cor:erm}) and polynomial activations (Corollary \ref{cor:poly}).  

\begin{theorem}\label{master_theorem}
For any continuous activation function $\sigma$ and r.v. $\bX\in \mathcal{R}_2(\sigma,n)$ with finite upper intrinsic dimension $\mathrm{dim}^*(\sigma,\bX)<\infty$, the loss function $$L(\btheta) = \expval\parq*{\ell\parr*{\bPhi(\bX;\btheta),\bY}}$$
for one-hidden-layer NNs $\bPhi(\bx;\btheta) = \bU\sigma( \bW\bx)$ admits no spurious valleys in the over-parame- trised regime $p\geq \mathrm{dim}^*(\sigma,\bX)$.
\end{theorem}

\paragraph{Sketch of the proof}
The proof consists of showing that we can construct a descent path verifying property \ref{p:conn} starting from any parameters $\btheta$. The construction can be articulated in two main parts. First, we show that we can map the starting parameter $\btheta_0 = (\bU_0,\bW_0)$ to another parameter $\btheta_{1/2} = (\bU_{1/2},\bW_{1/2})$ such that the functions $\bra*{\bx \mapsto \sigma(\prodscal{\bw_{1/2,i},\bx})}_{i\in[p]}$ form a basis of $V_\sigma$. It follows that there exists a minimal function $\bfun \in V_\sigma^m \doteq \bra*{ (f_1,\dots,f_m) \st f_i \in V_\sigma }$, i.e.
$$
\bfun \in \argmin_{\bg \in V_\sigma^m} \expval\parq*{ \ell\parr*{\bg(\bX),\bY} }
$$
which can be represented as $\bfun = \bPhi(\cdot;\btheta_1 = (\bU_1,\bW_{1/2}))$ for some $\bU_1$. The second part of the path can be thus taken as $t\mapsto (1-t)\bU_{1/2} + t\bU_1$: as the loss function is convex, this is descent path.

\paragraph{}

The above result can be interpreted as follows: if the network is such that any of its output units $\Phi_i$ can be chosen from the whole linear space spanned by its filter functions $V_\sigma$, then the associated optimization problem is such that there always exists a descent path to an optimal solution, for any initialization of the parameters.

Applying the observations in Section \ref{sec:dimension} describing the cases of finite intrinsic dimension, we immediately get the following corollaries.

\begin{corollary}[ERM]\label{cor:erm}
Consider $N$ data points $\bra{(\bx_i,\by_i)}_{i=1}^N \subset \mathbb{R}^n\times\mathbb{R}^m$. For one-hidden-layer NNs $\bPhi(\bx;\btheta) = \bU\sigma(\bW\bx)$, where $\sigma$ is any continuous activation function, the empirical loss function 
\begin{equation*}
L(\btheta) =\frac{1}{N} \sum_{i=1}^N \ell\parr*{\bPhi(\bx_i;\btheta),\by_i}    
\end{equation*}
admits no spurious valleys in the over-parametrized regime $p \geq N$.
\end{corollary}

\paragraph{Comparison with existing results}

This results was already shown in \citep{livni2014computational}. The only difference with our result is that we allow for rank degeneracy in the matrix $\sigma\parr*{\bW \parq*{\bx_1 |\cdots|\bx_N}}$. However,
its proof illustrates the danger of studying empirical risk minimization landscapes in over-parametrised regimes, since it bypasses all the geometric and algebraic properties needed in the population risk setting - which may be more relevant to understand the generalization properties of the model.

Other works considered the landscape of empirical risk minimization for deep networks. For ReLu-like activations, multi-layer networks and square losses, \citep{soudry2016no} showed that (almost surely) there exists no differentiable spurious minima if one of the layer weights $\bW_i \in \R^{p_i\times p_{i-1}}$ satisfy $p_ip_{i-1}\geq N$. \citep{nguyen2017loss} showed that no spurious minima occur for multi-layer NNs for a class of losses and activations, if one of the layers inner width exceeds the number of data points and the critical points verify certain non-degeneracy conditions. 


\begin{corollary}[Polynomial activations]\label{cor:poly}
For one-hidden-layer NNs $\bPhi(\bx;\btheta) = \bU\sigma(\bW\bx)$ with polynomial activation function $\sigma(z) = a_0 + a_1z + \cdots + a_dz^d$, the loss function $L(\btheta) = \expval\parq*{\ell\parr*{\bPhi(\bX;\btheta),\bY}}$ admits no spurious valleys in the over-parametrized regime 
$$
p\geq \sum_{i=1}^d\binom{n + i -1}{i}\mathbf{1}_{\bra{a_i\neq 0}} = O(n^d)
$$
\end{corollary}
Under the hypothesis of Corollary \ref{cor:poly} with $p = O(n^d)$, a generic function of $V_\sigma$, $\bPhi(\bx;\btheta) = \bu^T\sigma(\bW \bx)$, can be also represented, for some $\bgamma = \bgamma(\btheta)$, in the generalized linear form 
$$
\bPhi(\bx;\btheta) = \prodscal{\bgamma, \bvarphi(\bx)}
$$
with $\bvarphi(\bx) = (x_{k_1}\cdots x_{k_j})_{\bra{1\leq k_1\leq \dots \leq k_j \leq n, j \in [d]}}$. The parameters $\btheta$ and $\bgamma$ differ for their dimensions:
$$
\dim(\bgamma) = O(n^d) < \dim(\btheta) = (n+1) \cdot O(n^d) = O(n^{d+1})
$$
One would therefore like Corollary \ref{cor:poly} to hold also (at least) for $p \geq O(n^{d-1})$. In the next section we address this problem for the linear activation $\sigma(z) = z$ and the quadratic activation $\sigma(z) = z^2$.


\subsection{Improved over-parametrization bounds for homogeneous polynomial activations}

The over-parametrization bounds obtained in Corollary \ref{cor:poly} are quite non-desiderable in practical applications. We show that they can indeed be improved, for the case of linear and quadratic networks. 

\subsubsection{Linear networks case}

Linear networks have been considered as a first order approximation of feed-forward multi-layers networks \citep{kawaguchi2016deep}. It was shown, in several works \citep{kawaguchi2016deep,brunatopology,yun2018small}, that, for linear networks of any depth
\begin{equation}\label{linear_NN_any_depth}
    \bPhi(\bx;\btheta) = \bW_{K+1}\cdots \bW_1 \bx
\end{equation}
with $\btheta = (\bW_{K+1},\bW_K,\dots,\bW_2,\bW_1)\in \mathbb{R}^{m\times p_{K}}\times \mathbb{R}^{p_{K}\times p_{K-1}}\times\cdots \mathbb{R}^{p_2\times p_1}\times \mathbb{R}^{p_1\times n}$, the loss function $\eqref{eq:lossFun}$ has no spurious local minima, if $\min_{i\in[K]} p_i \geq \min\bra{n,m}$. This corresponds exactly with over-parametrization regime in Corollary \ref{cor:poly}, for the case of one-hidden-layer networks. The following theorem improves on Corollary \ref{cor:poly} for the case of multi-layer linear networks, showing that no over-parametrisation is required in this case to avoid spurious valleys, for square loss functions.
\begin{theorem}[Linear networks]
\label{theo:linear}
For linear NNs \eqref{linear_NN_any_depth} of any depth $K\geq 1$ and of any layer widths $p_k \geq 1$, $k\in[K]$, and any input-output dimensions $n, m \geq 1$, the square loss function $L(\btheta) = \expval\norm*{\bPhi(\bX;\btheta)-\bY}^2$ admits no spurious valleys. 
\end{theorem}

\subsubsection{Quadratic networks case}\label{sec:quad_nets_no_valleys}

Quadratic activations $\sigma(z) = z^2$ have been considered in the literature \citep{livni2014computational,du2018power,soltanolkotabi2017theoretical} as second order approximation of general non-linear activations.
Corollary \ref{cor:poly} says that, if $p \geq n(n+1)/2$, the loss function \eqref{eq:lossFun} admits no spurious valleys. In the following theorem we relax the over-parametrisation requirement and show that $p > 2n$ is sufficient for the statement to hold, in the case of square loss functions and one dimensional output ($m=1$). 

\begin{theorem}[Quadratic networks]\label{theo:quadratic_p_geq_n}
For one-hidden-layer NNs $\Phi(\bx;\btheta) = \bu^T\sigma(\bW\bx)$ with quadratic activation function $\sigma(z) = z^2$ and one-dimensional output ($m=1$), the square loss function $L(\btheta) = \expval\abs*{\Phi(\bX;\btheta) - Y}^2$ admits no spurious valleys in the over-parametrised regime $p \geq 2n+1 = O(n)$. 
\end{theorem}

\paragraph{Sketch of the proof} 
The proof (reported in Section \ref{sec:no spurious valleys proofs}) consists in constructing a path satisfying property \ref{p:conn} and improves upon the proof of Theorem \ref{master_theorem} by leveraging the special linearized structure of the network for quadratic activation. For every parameter $\btheta = (\bu,\bW) \in \mathbb{R}^{p}\times \mathbb{R}^{p\times n}$, we can write
$$
\Phi(\bx;\btheta) = \sum_{i=1}^p u_i (\prodscal{\bw_i,\bx})^2 = \prodscal[\Big]{\sum_{i=1}^p u_i\bw_i\bw_i^T,\bx\bx^T}_F 
$$
We notice that $\Phi(\cdot;\btheta)$ can also be represented by a NN $\Phi(\cdot;\hat{\btheta})$ with 
$n$ hidden units; indeed, if $\sum_{i=1}^n \sigma_i \bv_i\bv_i^T$ is the SVD of $\sum_{i=1}^p u_i\bw_i\bw_i^T$, then $\Phi(\bx;\btheta) = \prodscal{\sum_{i=1}^n \sigma_i \bv_i\bv_i^T,\bx\bx^T}_F$. 
Therefore $p\geq n$ is sufficient to describe any element in $V_\sigma$. 
A path to the symmetric matrix defining the optimal network is then constructed by mapping the above decomposition defined by the standard form of the network.

\paragraph{}
The factor $2$ in the statement is due to some technicalities in the proof, but a more involved proof should be able to extend the result to the regime $p\geq n$. The extension of such mechanism for higher order tensors (appearing as a result of multiple layers or high-order polynomial activations) using tensor decomposition also seems possible and is left for future work.

\paragraph{Comparison with previous works}

The same optimization landscape has been considered in the works \citep{soltanolkotabi2017theoretical} and \citep{du2018power}. In the first work, the authors show absence of spurious minima for the case of $p\geq 2 n$ and of ERM (loss evaluated on $N$ data points), but for fixed output layer weights; under some assumption on the output layer weights, the result is shown to still hold for $p\geq n$, if $n\leq N \leq O(n^2)$. This last condition can be removed by considering the regularized loss with non-zero weight decay, as shown in \citep{du2018power}; in the same work, the authors also proved absence of spurious minima in the case $p<n$ and $p(p+1)\geq 2N$ for a  randomly regularized loss (with high probability).

By relaxing the statement to absence of spurious valleys, we showed that this holds for the square loss (both in population and ERM setting) and the optimisation problem over both layer weights if $p>2n$.

\subsubsection{Lower to upper intrinsic dimension gap}

As observed in Lemma \ref{lemma:lid} $\dim_*(\sigma(z)=z,n) = 1$ and $\dim_*(\sigma(z)=z^2,n) = n$ for all integer $n\geq 1$. Therefore, Theorem \ref{theo:linear} and Theorem \ref{theo:quadratic_p_geq_n} say that, for $\sigma(z) = z^k$, $k \in [2]$, and $m=1$, the square loss function $L(\btheta) = \expval\abs*{\Phi(\bX;\btheta) - Y}^2$ admits no spurious valleys in the over-parametrized regime $p \geq O(\dim_*(\sigma,n))$. We conjecture that this hold for any (sufficiently regular) activation function with finite intrinsic lower dimension.

\section{Infinite intrinsic dimension and presence of spurious valleys}\label{sec:spurious valleys exist}

This section is devoted to the construction of worst-case scenarios for non-over parametrised networks. The main result (Theorem \ref{theo:main_counter_ex}) essentially states that, for networks with width smaller than the lower intrinsic dimension defined above, spurious valleys can be created by choosing adversarial data distributions. We then show how this implies negative results for under-parametrized polynomial architectures and a large variety of architectures used in practice.

\begin{theorem}\label{theo:main_counter_ex}
Consider the square loss function $L(\btheta) = \expval\norm{\bPhi(\bX;\btheta)-\bY}^2$ for one-hidden-layer NNs $\bPhi(\bx;\btheta) = \bU\sigma(\bW\bx)$ with non-negative activation function $\sigma \geq 0$ such that $\sigma\in L^2(\mathbb{R},e^{-x^2}\,dx)$. If $p \leq \frac{1}{2}\mathrm{dim}_*(\sigma,n-1)$, then there exists a r.v. $(\bX,\bY)$ such that the square loss function $L$ admits spurious valleys. In particular, for any given $M>0$, the r.v. $\bY$ can be chosen in such a way
that there exists a (non-empty) open set $\Omega\subset\Theta$ such that
\begin{equation}\label{eq:bad_valley}
M/2+\min_{\btheta \in \Omega} L(\btheta)  \geq \sup_{\btheta\in\Omega}L(\btheta) \geq \min_{\btheta \in \Omega} L(\btheta) \geq M + \min_{\btheta \in \Theta}L(\btheta)    
\end{equation}
and any path $\btheta:[0,1]\to \Theta$ such that $\btheta_0 \in \Omega$ and $\btheta_1$ is a global minima verifies
\begin{equation}\label{eq:tall_valley}
\max_{t\in[0,1]}L(\btheta_t) \geq \min_{\btheta\in\Omega} L(\btheta) + M    
\end{equation}
\end{theorem}
Equation \eqref{eq:bad_valley} in Theorem \ref{theo:main_counter_ex} says that any local descent algorithm, if initialized in $\btheta_0 \in \Omega$, at its best it will only be able to produce a final parameter value which is at least $M$ far from optimality. Equation \eqref{eq:tall_valley} implies that any path starting from parameter belonging to $\Omega$ must `up-climb' at least $M/2$ in the loss value. In the following we refer to such property, as stated in Theorem \ref{theo:main_counter_ex}, by saying that \emph{the loss function has arbitrarily bad spurious valleys}.
Note that this result ensures that spurious valleys have positive Lebesgue measure, so there is a positive probability that gradient descent methods initialized with a measure that is absolutely continuous with respect to Lebesgue will get stuck in a bad local minima.

Applying the observations describing the values of the lower intrinsic dimension for different activation functions, we get the following corollaries. 

\begin{corollary}[Homogeneous even degree polynomial activations]\label{cor: spurious polynomial}
Consider the case of activation $\sigma(z) = z^{2k}$ with $k\geq 1$ integer. For one-hidden-layer NNs $\bPhi(\bx;\btheta) = \bU\sigma(\bW\bx)$, if $n\geq 2$ and the hidden layer width satisfies
\begin{equation*}
p \leq \begin{cases}n-1 & \text{if } k=1\\ \frac{1}{2}\rk_\mathrm{S}(2k,n-1) & \text{if } k > 1\end{cases}    
\end{equation*}
then there exists a r.v. $(\bX,\bY)$ such that the square loss function $L(\btheta) = \expval\norm{\bPhi(\bX;\btheta)-\bY}^2$ has arbitrarily bad spurious valleys. 
\end{corollary}
This follows by Theorem \ref{theo:main_counter_ex} and Lemma \ref{lemma:lid}, since $\mathrm{dim}_*(\sigma(z) = z^{2k},n) = \rk_\mathrm{S}(2k,n)$. 
For the well known case $k =1$ (symmetric matrices) it holds $\rk_\mathrm{S}(2,n) = n$; therefore Corollary \ref{cor: spurious polynomial} implies that the bound provided in Corollary \ref{cor:poly} is almost (up to a factor $2$) tight. 
Notice that our result is indeed in line with the results discussed in Section \ref{sec:quad_nets_no_valleys}.

\begin{corollary}[Spurious valleys exist in generic architectures]\label{cor: spurious generic}
If $n\geq 2$, for one-hidden-layer NNs $\bPhi(\bx;\btheta) = \bU\sigma(\bW\bx)$ with any hidden layer width $p \geq 1$ and continuous non-negative non-polynomial activation function $\sigma \in L^2(\mathbb{R},e^{-x^2/2})$, then there exists a r.v. $(\bX,\bY)$ such that the square loss function $L(\btheta) = \expval\norm{\bPhi(\bX;\btheta)-\bY}^2$ has arbitrarily bad spurious valleys. This setting includes the following activation functions: 
\begin{itemize}
\item The ReLU activation function $\sigma(z) = z_+$ and some relaxations of it, such as softplus activation functions $\sigma(z) = \beta^{-1}\log\parr*{1 + e^{\beta z}}$, with $\beta > 0$;
\item The sigmoid activation function $\sigma(z) = (1 + e^{-z})^{-1}$ and the approximating erf function $\sigma(z) = 2/\pi \int_0^z e^{-u}\, du$, which represents an approximation to the sigmoid function.
\end{itemize}
\end{corollary}
This follows by Theorem \ref{theo:main_counter_ex} by observing that $\mathrm{dim}_*(\sigma,n) = \infty$ if $\sigma$ is one of the above activation functions. 

\paragraph{Discussion and comparison with previous works}

Several works showed existence of spurious minima: \citep{safran2017spurious} showed counterexamples under Gaussian input distributions, for $p = n-1 \in \bra{8,\dots,19}$, using a computer-assisted proof; \citep{swirszcz2016local} and \citep{zhou2017critical} provided a few numerical examples; \citep{yun2018small} showed existence of spurious minima for ReLU-like activations under non-realizability, and provided counterexamples for smooth activations. For any number of hidden neurons $p$, we give a (constructive) proof of existence of a data distribution which creates spurious valleys, under the only assumption of non-negative continuous activation function. We also remark that while in the above works the authors proved existence of spurious local minima, we prove that, in fact, arbitrarily bad spurious valleys can exist, which is a stronger negative characterization.


The results of this section can be interpreted as worst-case scenarios for the problem of optimizing \eqref{eq:lossFun}. We showed that, even for simple one-hidden-layer neural network architectures with non-linear activation functions used in practice (such as ReLU), global optimality results can not hold, unless we make some assumptions on the data distributions.

\section{Typical spurious valleys and low-energy barriers}\label{sec:gap}

In the previous section it was shown that whenever the
number of hidden units $p$ is below the lower intrinsic dimension, then 
one can show worst-case data distributions that yield a landscape with arbitrarily 
bad spurious valleys. A natural follow-up question is thus to consider the complexity of the energy landscape in a \emph{typical} scenario, defined in terms of both parameter initialisation (how likely are descent algorithms to fall into a spurious valley?) and energy value (how deep are typical spurious valleys?). 

In this section, we study the energy landscape under generic data distributions in case of homogeneous activation, and
show that, although spurious valleys may appear, they tend do so below a certain energy level,
controlled by the decay of the spectral decomposition of the kernel defined by the activation function and by the amount of parametrisation $p$. 
This offers a first glimpse at the empirical success of local descent algorithms in 
conditions where $p$ is indeed below the intrinsic dimension. 

We consider 
oracle square loss functions of the form
\begin{equation}\label{eq:mse}
L(\btheta) = \E
\abs*{\Phi(\bX;\btheta) - Y }^2 
\end{equation}
for one-dimensional output one-hidden-layer NNs $\Phi(\bx;\btheta) = \bu^T \sigma(\bW \bx)$, with $\btheta = (\bu, \bW) \in \mathbb{R}^p \times \mathbb{R}^{p\times n}$, $\sigma$ a positively homogeneous function,  
and $\bX,Y$ square integrable r.v. 
Notice that we can write
\begin{equation*}
L(\btheta) = \E\abs*{\Phi(\bX;\btheta) - f^*(\bX)}^2 + \E\abs*{Y - f^*(\bX)}^2
\end{equation*}
for some measurable $f^*:\mathbb{R}^n\to\mathbb{R}$ such that $f^*(\bX) = \E\parq*{Y\,|\,\bX}$. In particular this implies that
\begin{equation*}
\min_{\btheta\in\Theta}L(\btheta) \geq \mathcal{R}\parr{\bX,Y} \doteq \E\abs*{Y - f^*(\bX)}^2 
\end{equation*}
If $f^*$ can be written as a  one-hidden-layer neural network with an arbitrary number of hidden units, that is 
$$
f^*(\bx) = \int_{\R^n} \sigma(\prodscal{\bx, \bw})\rho(\bw)\,d\mu(\bw)
$$
for some measure $\mu$ and weight function $\rho$, then a possible approach to find a proper approximation of $f^*$ is through random features sampling  \citep{rahimi2008random}. Applying some recent results \citep{bach2017equivalence} relating random features expansions with kernel quadrature rules, we show that this implies the following statement:
\emph{as the network width increases, spurious valleys tend to be confined to decreasingly low loss value}. In this regime, large loss barriers are therefore avoided with high probability over initialization of the parameters. The statement is made more rigorous in the following:



\begin{theorem}\label{theo:energy_gap}
Let $d\tau$ be the uniform distribution over the unit sphere $\mathbb{S}^n$ and consider an initial parameter 
$\tilde{\btheta} = (\tilde{\bu},\tilde{\bW})$
with 
$\tilde{\bw}_i \sim d\tau$
sampled i.i.d. Then the following hold:
\begin{enumerate}
\item There exists a path $t\in[0,1]\mapsto \btheta_t$ such that $\btheta_0 = \tilde{\btheta}$, the function $t\in[0,1]\mapsto L(\btheta_t)$ is non-increasing, and 
$$
L(\btheta_1) \leq \mathcal{R}(\bX,Y) + \lambda \quad\text{if}\quad p \geq O\parr*{-\lambda^{-1}\log(\lambda\delta)}
$$
with probability greater or equal then $1-\delta$, for every $\lambda,\delta \in (0,1)$.
\item If $f^*$ is sufficiently regular\footnote{More precisely, if the function $f^*$ can be written as 
$f^*(\bx) = \int_\mathcal{W} g^*(\bw)\psi_\bw(\bx)\,d\tau(\bw)$ for some $g^* \in L^\infty_{d\tau}$.}, there exists a path $t\in[0,1]\mapsto \btheta_t$ such that $\btheta_0 = \tilde{\btheta}$, the function $t\in[0,1]\mapsto L(\btheta_t)$ is non-increasing, and 
$$
L(\btheta_1) \leq \mathcal{R}(\bX,Y) + O(p^{-1+\delta})
$$
with probability greater or equal then $1-e^{-O(p^\delta)}$ for every $\delta\in(0,1)$.
\end{enumerate}
\end{theorem}


\paragraph{Sketch of the proof}

Assume that $f^*$ admits the representation
$$
f^*(\bx) = \int_\Theta \rho(\bw)\sigma(\prodscal{\bx,\bw})\,d\tau(\bw)
$$
for some density $\rho$. If $\bw_i \sim d\tau$, $i\in[p]$, are drawn i.i.d., we have
$$
\expval\parr*{\frac{1}{p}\sum_{i=1}^p\rho(\bw_i)\sigma(\prodscal{\bw_i,\bx}) - f^*(\bx)}^2 = O\parr*{\frac{1}{p}}
$$
Notice that by only moving the second layer, we can construct a (linear) descent path from $(\tilde{\bu},\tilde{\bW})$ to $(\bu,\tilde{\bW})$, where $u_i = \rho(\bw_i)$. The proof is then concluded by applying an Hoeffding's-type inequality  to get property 2. if it holds $\rho \in L^\infty(\mathbb{S}^n\,d\tau)$ or by applying Proposition 1 in \citep{bach2017equivalence} to obtain property 1.

\paragraph{Related works}

Many recent works leveraged arguments based
on random features to explain the empirical success of local descent algorithms to train neural networks (see e.g. \citep{jacot2018neural, allen2018convergence, oymak2019towards, yehudai2019power, ma2019comparative,du2018gradient}). In Theorem \ref{theo:energy_gap}, we used this type of technique to show properties of the optimization landscape. The main limitation shared by our and the cited results is the gap between the regimes in which the apply (high over-parametrized NNs) and the regimes attained in practice. A current important direction is to understand the dynamics of neural networks training over kernel approximation and to extend such results to \emph{moderatly} over-parametrized architectures.

\begin{remark}
Notice that in the previous description of the problem, we dropped bias terms from the neural network architectures for sake of simplicity, as we can immediately generalize to the biases case by stacking the bias in the weights and input random variables.
As a bias term is needed in order to use universal approximation results, with abuse of notation, in the above theorem we wrote $\prodscal{\bw,\bx} \doteq \prodscal{\bw^{(n)},\bx} + w_{n+1}$ for $\bw\in\mathbb{S}^n, \bx\in\mathbb{R}^n$, where $\bw^{(n)} = (w_1,\dots,w_n)$ represents a neuron weight and $w_{n+1}$ a bias term; again, note that this can be done by simply considering the r.v. $\tilde{\bX} \doteq (\bX,1)$ in place of $\bX$.
\end{remark}





\section{Future directions}\label{section:discussion}

We considered the problem of characterizing the loss surface of neural networks from the perspective of optimization, with the goal of deriving weak certificates that enable - or prevent - the existence of descent paths towards global minima. 


The topological properties studied in this paper, however, do not yet capture fundamental aspects that are necessary to explain the empirical success of deep learning methods. We identify a number of different directions that deserve further attention.



The positive results presented above rely on being able to reduce the network to the case when (convex) optimization over the output layer is sufficient to reach optimal weight values. A better understanding of first layer dynamics needs to be carried out.
Moreover, in such positive results we only proved non-existence of (high) energy barriers. While this is an interesting property from the optimization point of view, it is also not sufficient to guarantee convergence of local descent algorithms. Another informative property of the loss function that should be addressed in future works is the existence of local descents in non optimal points: for every $\btheta_0\in\Theta$ non optimal and any neighborhood $\mathcal{U}\subseteq \Theta$ of $\btheta_0$, there exists $\btheta\in\mathcal{U}$ such that $L(\btheta)<L(\btheta_0)$. More generally, our present work is not informative on the performance of gradient descent in the regimes with no spurious valley. 


The other very important point to be addressed in future is how to extend the above results to architectures of more practical interest. Depth 
and the specific linear structure of Convolutional Neural Networks, critical to explain the excellent empirical performance of deep learning in computer vision, text or speech, need to be exploited, as well as specific design choices such as Residual connections  and several normalization strategies -- as done recently in \citep{shamir2018resnets} and \citep{santurkar2018does} respectively. This also requires making specific assumptions on the data distribution, and is left for future work. 

\paragraph{Acknowledgements} We would like to thank G\'erard Ben Arous and L\'eon Bottou
for fruitful discussions, and Jean Ponce for valuable comments and corrections of the original version of this manuscript. LV would also like to thank Jumageldi Charyyev for fruitful discussions on the proofs of several propositions and Andrea Ottolini for valuable comments on a previous version of this manuscript. LV was partially supported by NSF
grant DMS-1719545. ASB was partially supported by NSF
grants DMS-1712730 and DMS-1719545, and by a grant from the Sloan
Foundation. JB acknowledges the partial support by the Alfred P. Sloan Foundation, NSF RI-1816753, NSF CAREER CIF 1845360, and Samsung Electronics.

\bibliography{references}

\begin{thebibliography}{47}
\providecommand{\natexlab}[1]{#1}
\providecommand{\url}[1]{\texttt{#1}}
\expandafter\ifx\csname urlstyle\endcsname\relax
  \providecommand{\doi}[1]{doi: #1}\else
  \providecommand{\doi}{doi: \begingroup \urlstyle{rm}\Url}\fi

\bibitem[Allen-Zhu et~al.(2018)Allen-Zhu, Li, and Song]{allen2018convergence}
Zeyuan Allen-Zhu, Yuanzhi Li, and Zhao Song.
\newblock A convergence theory for deep learning via over-parameterization.
\newblock \emph{arXiv preprint arXiv:1811.03962}, 2018.

\bibitem[Bach(2017{\natexlab{a}})]{bach2017breaking}
Francis Bach.
\newblock Breaking the curse of dimensionality with convex neural networks.
\newblock \emph{Journal of Machine Learning Research}, 18\penalty0
  (19):\penalty0 1--53, 2017{\natexlab{a}}.

\bibitem[Bach(2017{\natexlab{b}})]{bach2017equivalence}
Francis Bach.
\newblock On the equivalence between kernel quadrature rules and random feature
  expansions.
\newblock \emph{Journal of Machine Learning Research}, 18\penalty0
  (21):\penalty0 1--38, 2017{\natexlab{b}}.

\bibitem[Bandeira et~al.(2016)Bandeira, Boumal, and
  Voroninski]{bandeira2016low}
Afonso~S Bandeira, Nicolas Boumal, and Vladislav Voroninski.
\newblock On the low-rank approach for semidefinite programs arising in
  synchronization and community detection.
\newblock In \emph{Conference on Learning Theory}, pages 361--382, 2016.

\bibitem[Bottou et~al.(2016)Bottou, Curtis, and
  Nocedal]{bottou2016optimization}
L{\'e}on Bottou, Frank~E Curtis, and Jorge Nocedal.
\newblock Optimization methods for large-scale machine learning.
\newblock \emph{arXiv preprint arXiv:1606.04838}, 2016.

\bibitem[Bottou et~al.(2017)Bottou, Arjovsky, Lopez-Paz, and
  Oquab]{bottou2017geometrical}
Leon Bottou, Martin Arjovsky, David Lopez-Paz, and Maxime Oquab.
\newblock Geometrical insights for implicit generative modeling.
\newblock \emph{arXiv preprint arXiv:1712.07822}, 2017.

\bibitem[Boucheron et~al.(2013)Boucheron, Lugosi, and
  Massart]{boucheron2013concentration}
St{\'e}phane Boucheron, G{\'a}bor Lugosi, and Pascal Massart.
\newblock \emph{Concentration inequalities: A nonasymptotic theory of
  independence}.
\newblock Oxford university press, 2013.

\bibitem[Boumal et~al.(2016)Boumal, Voroninski, and Bandeira]{boumal2016non}
Nicolas Boumal, Vlad Voroninski, and Afonso Bandeira.
\newblock The non-convex burer-monteiro approach works on smooth semidefinite
  programs.
\newblock In \emph{Advances in Neural Information Processing Systems}, pages
  2757--2765, 2016.

\bibitem[Choromanska et~al.(2015)Choromanska, Henaff, Mathieu, Arous, and
  LeCun]{choromanska2015loss}
Anna Choromanska, Mikael Henaff, Michael Mathieu, G{\'e}rard~Ben Arous, and
  Yann LeCun.
\newblock The loss surfaces of multilayer networks.
\newblock In \emph{Artificial Intelligence and Statistics}, pages 192--204,
  2015.

\bibitem[Comon et~al.(2008)Comon, Golub, Lim, and Mourrain]{comon2008symmetric}
Pierre Comon, Gene Golub, Lek-Heng Lim, and Bernard Mourrain.
\newblock Symmetric tensors and symmetric tensor rank.
\newblock \emph{SIAM Journal on Matrix Analysis and Applications}, 30\penalty0
  (3):\penalty0 1254--1279, 2008.

\bibitem[Daniely et~al.(2016)Daniely, Frostig, and Singer]{daniely2016toward}
Amit Daniely, Roy Frostig, and Yoram Singer.
\newblock Toward deeper understanding of neural networks: The power of
  initialization and a dual view on expressivity.
\newblock In \emph{Advances In Neural Information Processing Systems}, pages
  2253--2261, 2016.

\bibitem[Dauphin et~al.(2014)Dauphin, Pascanu, Gulcehre, Cho, Ganguli, and
  Bengio]{dauphin2014identifying}
Yann~N Dauphin, Razvan Pascanu, Caglar Gulcehre, Kyunghyun Cho, Surya Ganguli,
  and Yoshua Bengio.
\newblock Identifying and attacking the saddle point problem in
  high-dimensional non-convex optimization.
\newblock In \emph{Advances in neural information processing systems}, pages
  2933--2941, 2014.

\bibitem[Draxler et~al.(2018)Draxler, Veschgini, Salmhofer, and
  Hamprecht]{draxler2018essentially}
Felix Draxler, Kambis Veschgini, Manfred Salmhofer, and Fred~A Hamprecht.
\newblock Essentially no barriers in neural network energy landscape.
\newblock \emph{arXiv preprint arXiv:1803.00885}, 2018.

\bibitem[Du and Lee(2018)]{du2018power}
Simon~S Du and Jason~D Lee.
\newblock On the power of over-parametrization in neural networks with
  quadratic activation.
\newblock \emph{arXiv preprint arXiv:1803.01206}, 2018.

\bibitem[Du et~al.(2017)Du, Lee, Tian, Poczos, and Singh]{du2017gradient}
Simon~S Du, Jason~D Lee, Yuandong Tian, Barnabas Poczos, and Aarti Singh.
\newblock Gradient descent learns one-hidden-layer cnn: Don't be afraid of
  spurious local minima.
\newblock \emph{arXiv preprint arXiv:1712.00779}, 2017.

\bibitem[Du et~al.(2018)Du, Lee, Li, Wang, and Zhai]{du2018gradient}
Simon~S Du, Jason~D Lee, Haochuan Li, Liwei Wang, and Xiyu Zhai.
\newblock Gradient descent finds global minima of deep neural networks.
\newblock \emph{arXiv preprint arXiv:1811.03804}, 2018.

\bibitem[Feizi et~al.(2017)Feizi, Javadi, Zhang, and Tse]{feizi2017porcupine}
Soheil Feizi, Hamid Javadi, Jesse Zhang, and David Tse.
\newblock Porcupine neural networks:(almost) all local optima are global.
\newblock \emph{arXiv preprint arXiv:1710.02196}, 2017.

\bibitem[Freeman and Bruna(2017)]{brunatopology}
Daniel Freeman and Joan Bruna.
\newblock Topology and geometry of half-rectified network optimization.
\newblock \emph{ICLR 2017}, 2017.

\bibitem[Ge et~al.(2016)Ge, Lee, and Ma]{ge2016matrix}
Rong Ge, Jason~D Lee, and Tengyu Ma.
\newblock Matrix completion has no spurious local minimum.
\newblock In \emph{Advances in Neural Information Processing Systems}, pages
  2973--2981, 2016.

\bibitem[Ge et~al.(2017{\natexlab{a}})Ge, Jin, and Zheng]{ge2017no}
Rong Ge, Chi Jin, and Yi~Zheng.
\newblock No spurious local minima in nonconvex low rank problems: A unified
  geometric analysis.
\newblock \emph{arXiv preprint arXiv:1704.00708}, 2017{\natexlab{a}}.

\bibitem[Ge et~al.(2017{\natexlab{b}})Ge, Lee, and Ma]{ge2017learning}
Rong Ge, Jason~D Lee, and Tengyu Ma.
\newblock Learning one-hidden-layer neural networks with landscape design.
\newblock \emph{arXiv preprint arXiv:1711.00501}, 2017{\natexlab{b}}.

\bibitem[Gunasekar et~al.(2018)Gunasekar, Lee, Soudry, and
  Srebro]{gunasekar2018characterizing}
Suriya Gunasekar, Jason Lee, Daniel Soudry, and Nathan Srebro.
\newblock Characterizing implicit bias in terms of optimization geometry.
\newblock \emph{arXiv preprint arXiv:1802.08246}, 2018.

\bibitem[Hardt and Ma(2016)]{hardt2016identity}
Moritz Hardt and Tengyu Ma.
\newblock Identity matters in deep learning.
\newblock \emph{arXiv preprint arXiv:1611.04231}, 2016.

\bibitem[Hornik(1991)]{hornik1991approximation}
Kurt Hornik.
\newblock Approximation capabilities of multilayer feedforward networks.
\newblock \emph{Neural networks}, 4\penalty0 (2):\penalty0 251--257, 1991.

\bibitem[Jacot et~al.(2018)Jacot, Gabriel, and Hongler]{jacot2018neural}
Arthur Jacot, Franck Gabriel, and Cl{\'e}ment Hongler.
\newblock Neural tangent kernel: Convergence and generalization in neural
  networks.
\newblock In \emph{Advances in neural information processing systems}, pages
  8571--8580, 2018.

\bibitem[Ji and Telgarsky(2018)]{ji2018risk}
Ziwei Ji and Matus Telgarsky.
\newblock Risk and parameter convergence of logistic regression.
\newblock \emph{arXiv preprint arXiv:1803.07300}, 2018.

\bibitem[Kawaguchi(2016)]{kawaguchi2016deep}
Kenji Kawaguchi.
\newblock Deep learning without poor local minima.
\newblock In \emph{Advances in Neural Information Processing Systems}, pages
  586--594, 2016.

\bibitem[Leshno et~al.(1993)Leshno, Lin, Pinkus, and
  Schocken]{leshno1993multilayer}
Moshe Leshno, Vladimir~Ya Lin, Allan Pinkus, and Shimon Schocken.
\newblock Multilayer feedforward networks with a nonpolynomial activation
  function can approximate any function.
\newblock \emph{Neural Networks}, 6\penalty0 (6):\penalty0 861--867, 1993.

\bibitem[Livni et~al.(2014)Livni, Shalev-Shwartz, and
  Shamir]{livni2014computational}
Roi Livni, Shai Shalev-Shwartz, and Ohad Shamir.
\newblock On the computational efficiency of training neural networks.
\newblock In \emph{Advances in Neural Information Processing Systems}, pages
  855--863, 2014.

\bibitem[Ma et~al.(2019)Ma, Wu, et~al.]{ma2019comparative}
Chao Ma, Lei Wu, et~al.
\newblock A comparative analysis of the optimization and generalization
  property of two-layer neural network and random feature models under gradient
  descent dynamics.
\newblock \emph{arXiv preprint arXiv:1904.04326}, 2019.

\bibitem[Mondelli and Montanari(2018)]{mondelli2018connection}
Marco Mondelli and Andrea Montanari.
\newblock On the connection between learning two-layers neural networks and
  tensor decomposition.
\newblock \emph{arXiv preprint arXiv:1802.07301}, 2018.

\bibitem[Nguyen and Hein(2017)]{nguyen2017loss}
Quynh Nguyen and Matthias Hein.
\newblock The loss surface of deep and wide neural networks.
\newblock \emph{arXiv preprint arXiv:1704.08045}, 2017.

\bibitem[Oymak and Soltanolkotabi(2019)]{oymak2019towards}
Samet Oymak and Mahdi Soltanolkotabi.
\newblock Towards moderate overparameterization: global convergence guarantees
  for training shallow neural networks.
\newblock \emph{arXiv preprint arXiv:1902.04674}, 2019.

\bibitem[Pennington and Bahri(2017)]{pennington2017geometry}
Jeffrey Pennington and Yasaman Bahri.
\newblock Geometry of neural network loss surfaces via random matrix theory.
\newblock In \emph{International Conference on Machine Learning}, pages
  2798--2806, 2017.

\bibitem[Rahimi and Recht(2008)]{rahimi2008random}
Ali Rahimi and Benjamin Recht.
\newblock Random features for large-scale kernel machines.
\newblock In \emph{Advances in neural information processing systems}, pages
  1177--1184, 2008.

\bibitem[Safran and Shamir(2016)]{safran2016quality}
Itay Safran and Ohad Shamir.
\newblock On the quality of the initial basin in overspecified neural networks.
\newblock In \emph{International Conference on Machine Learning}, pages
  774--782, 2016.

\bibitem[Safran and Shamir(2017)]{safran2017spurious}
Itay Safran and Ohad Shamir.
\newblock Spurious local minima are common in two-layer relu neural networks.
\newblock \emph{arXiv preprint arXiv:1712.08968}, 2017.

\bibitem[Santurkar et~al.(2018)Santurkar, Tsipras, Ilyas, and
  Madry]{santurkar2018does}
Shibani Santurkar, Dimitris Tsipras, Andrew Ilyas, and Aleksander Madry.
\newblock How does batch normalization help optimization?(no, it is not about
  internal covariate shift).
\newblock \emph{arXiv preprint arXiv:1805.11604}, 2018.

\bibitem[Shamir(2018)]{shamir2018resnets}
Ohad Shamir.
\newblock Are resnets provably better than linear predictors?
\newblock \emph{arXiv preprint arXiv:1804.06739}, 2018.

\bibitem[Soltanolkotabi et~al.(2017)Soltanolkotabi, Javanmard, and
  Lee]{soltanolkotabi2017theoretical}
Mahdi Soltanolkotabi, Adel Javanmard, and Jason~D Lee.
\newblock Theoretical insights into the optimization landscape of
  over-parameterized shallow neural networks.
\newblock \emph{arXiv preprint arXiv:1707.04926}, 2017.

\bibitem[Soudry and Carmon(2016)]{soudry2016no}
Daniel Soudry and Yair Carmon.
\newblock No bad local minima: Data independent training error guarantees for
  multilayer neural networks.
\newblock \emph{arXiv preprint arXiv:1605.08361}, 2016.

\bibitem[Soudry et~al.(2017)Soudry, Hoffer, and Srebro]{soudry2017implicit}
Daniel Soudry, Elad Hoffer, and Nathan Srebro.
\newblock The implicit bias of gradient descent on separable data.
\newblock \emph{arXiv preprint arXiv:1710.10345}, 2017.

\bibitem[Swirszcz et~al.(2016)Swirszcz, Czarnecki, and
  Pascanu]{swirszcz2016local}
Grzegorz Swirszcz, Wojciech~Marian Czarnecki, and Razvan Pascanu.
\newblock Local minima in training of neural networks.
\newblock \emph{arXiv preprint arXiv:1611.06310}, 2016.

\bibitem[Wilson et~al.(2017)Wilson, Roelofs, Stern, Srebro, and
  Recht]{wilson2017marginal}
Ashia~C Wilson, Rebecca Roelofs, Mitchell Stern, Nati Srebro, and Benjamin
  Recht.
\newblock The marginal value of adaptive gradient methods in machine learning.
\newblock In \emph{Advances in Neural Information Processing Systems}, pages
  4148--4158, 2017.

\bibitem[Yehudai and Shamir(2019)]{yehudai2019power}
Gilad Yehudai and Ohad Shamir.
\newblock On the power and limitations of random features for understanding
  neural networks.
\newblock \emph{arXiv preprint arXiv:1904.00687}, 2019.

\bibitem[Yun et~al.(2018)Yun, Sra, and Jadbabaie]{yun2018small}
Chulhee Yun, Suvrit Sra, and Ali Jadbabaie.
\newblock Small nonlinearities in activation functions create bad local minima
  in neural networks.
\newblock \emph{arXiv preprint arXiv:1802.03487}, 2018.

\bibitem[Zhou and Liang(2017)]{zhou2017critical}
Yi~Zhou and Yingbin Liang.
\newblock Critical points of neural networks: Analytical forms and landscape
  properties.
\newblock \emph{arXiv preprint arXiv:1710.11205}, 2017.

\end{thebibliography}

\appendix

\section{Proofs of Section \ref{sec:no spurious valleys}}
\label{sec:no spurious valleys 
proofs}

\paragraph{Notations} 

For any r.v.'s $\bX$ and $\bY$ with values in $\mathbb{R}^n$ and $\mathbb{R}^m$ respectively, we denote $\bSigma_\bX  = \expval\parq*{\bX\bX^T}$ and $\bSigma_{\bX\bY}  = \expval\parq*{\bX\bY^T}$. For every integer $n \geq 1$, we denote by $GL(n)$, $O(n)$ and $SO(n)$, respectively, the general linear group, the orthogonal group and the special orthogonal group of real $n\times n$ matrices. $\bI$ denotes the identity matrix and $\be_1,\dots,\be_n$ the standard basis in $\mathbb{R}^n$.

\subsection{Proof of Theorem \ref{master_theorem}}

We note that, under the assumptions of Theorem \ref{master_theorem}, the same optimal NN functions $\Phi_i(\cdot;\btheta)$ could also be obtained using a generalized linear model, where the representation function has the linear form $\Phi_i(\bx;\btheta) = \prodscal{\btheta_i,\bvarphi(\bx)}$, for some parameter independent function $\bvarphi:\mathbb{R}^n \to \mathbb{R}^{\mathrm{dim}^*(\sigma,\bX)}$. The main difference between the two models is that the former requires the choice of a non-linear activation function $\sigma$, while the latter implies the choice of a kernel functions. This is the content of the following lemma.

\begin{lemma}\label{lemma:intrinsic_subspace}
Let $\sigma:
\mathbb{R}\to\mathbb{R}$ be a continuous function and $\bX \in \mathcal{R}_2(\sigma,n)$ a r.v. Assume that the linear space
\begin{equation*}
V_{\sigma,\bX} \doteq \mathrm{span}\parr*{\bra{f \st f \in V_{\sigma,1}}} \subseteq L^2_\bX
\end{equation*}
is finite dimensional. Then there exists a scalar product $\prodscal{\cdot,\cdot}$ on $V_{\sigma,\bX}$ and a map $\bx\in\mathbb{R}^n\mapsto\bvarphi(\bx)\in V_{\sigma,\bX}$ such that
\begin{equation}\label{scalar_product_rkhs_lemma}
    \prodscal{\bpsi_{\sigma,\bw},\bvarphi(\bx)} = \psi_{\sigma,\bw}(\bx) = \sigma(\prodscal{\bw,\bx})
\end{equation}
for all $\bw\in\mathbb{R}^n$. Moreover, the function $\bw\in\mathbb{R}^n\mapsto \bpsi_{\sigma,\bw}\in V_{\sigma,\bX}$ is continuous.
\end{lemma}
\begin{proof}
For sake of simplicity, in the following we write $\psi_\bw$ for $\psi_{\sigma,\bw}$ and $V$ for $V_{\sigma,\bX}$. Let $\bpsi_{\bw_1},\dots,\bpsi_{\bw_q}$ be a basis of $V$. If $\bpsi_\bw = \sum_{i=1}^q\alpha_i\bpsi_{\bw_i}$ and $\psi_\bv = \sum_{j=1}^q\beta_j\bpsi_{\bw_j}$, then we can define a scalar product on $V$ as
\begin{equation*}
\prodscal{\bpsi_\bw,\bpsi_\bv} \doteq \sum_{i=1}^q\alpha_i\beta_i
\end{equation*}
If we define the map $\bx\in\mathbb{R}^n\mapsto\bvarphi(\bx)\in V$ as
\begin{equation*}
    \bvarphi(\bx) = \sum_{i=1}^q\psi_{\bw_i}(\bx) \bpsi_{\bw_i}
\end{equation*}
then property \eqref{scalar_product_rkhs_lemma} follows directly by the definition of the function $\psi_\bw$. Moreover, we can choose $\bx_1,\dots,\bx_q$ such that $\bvarphi(\bx_1),\dots,\bvarphi(\bx_q)$ is a basis of $V$. Now we need to show that, for $i\in[q]$, the map $\bw\mapsto \prodscal{\bpsi_\bw,\bpsi_{\bw_i}}$ is continuous. Let $\bM$ be the matrix $\bM \doteq (\psi_{\bw_j}(\bx_i))_{i,j} \in \mathbb{R}^{q\times q}$ and $\bz(\bw)$ be the vector $\bz(\bw)\doteq (\psi_\bw(\bx_i))_i \in\mathbb{R}^q$. Then $\prodscal{\bpsi_\bw,\bpsi_{\bw_i}} = (\bM^{-1}\bz(\bw))_i$, which is continuous in $\bw$. This shows that the map $\bw\in\mathbb{R}^n\mapsto \bpsi_\bw\in V$ is continuous.
\end{proof}

The non-trivial fact captured by Theorem \ref{master_theorem} is the following: when the capacity of network is large enough to match a generalized linear model, but still finite, then the problem of optimizing the loss function (\ref{eq:lossFun}), which is in general a highly non-convex object, satisfies an interesting optimization property in view of the local descent algorithms which are used in practice to solve it.

\begin{proof}[Proof of Theorem \ref{master_theorem}]
Thanks to Lemma \ref{lemma:intrinsic_subspace}, there exist two continuous maps $\bvarphi,\bpsi : \mathbb{R}^n \to \mathbb{R}^q \simeq V_{\sigma,\bX}$, with $q = \mathrm{dim}^*(\sigma,\bX)$, such that $\sigma(\prodscal{\bw,\bx}) = \prodscal{\bpsi(\bw),\bvarphi(\bx)}$ for every $\bw,\bx \in \mathbb{R}^n$. Therefore, every one-hidden-layer NN
$\bPhi(\bx;\btheta) = \bU\sigma(\bW\bx)$ can be written as $\bPhi(\bx;\btheta) = \bU\bpsi(\bW)\bvarphi(\bx)$, where, if $\bW \in \mathbb{R}^{p \times n}$, then $\bpsi(\bW)  \in \mathbb{R}^{p \times q}$ (that is $\bpsi$ is applied row-wise). 

The proof of the Theorem consists in exploiting the above \emph{linearized} representation of $\bPhi$ to show that property \ref{p:conn} holds (remind that this is equivalent to saying that the loss function has no spurious valleys).
Given an initial parameter $\tilde{\btheta} = (\tilde{\bU},\tilde{\bW})$, we want to construct a continuous path $t\in[0,1]\mapsto \btheta_t = (\bU_t,\bW_t)$, such that the function $t \in [0,1] \mapsto L(\btheta_t)$ is non-increasing and such that $\btheta_0 = \tilde{\btheta}$, $\btheta_1 \in \argmin_{\btheta} L(\btheta)$, where $L(\btheta) = \expval\parq*{ \ell(\bPhi(\bX;\btheta),\bY) }$. The construction of such a path can be articulated in two main steps:

\paragraph{Step 1.}
The first part of the path consist showing that we can assume that $\rk(\bpsi(\tilde{\bW})) = q$ w.l.o.g. Let  $\bw_1^T,\dots,\bw_p^T\in\mathbb{R}^n$ be the rows of $\tilde{\bW}$; suppose that $\rk(\bpsi(\tilde{\bW})) = r < q$ (otherwise there is nothing to show) and that $\bpsi(\bw_{i_1}),\dots,\bpsi(\bw_{i_r})$ are linearly independent. Denote $I =\bra{i_1,\dots,i_r}$, $J = [1,p]\setminus I = \bra{j_1,\dots,j_{p-r}}$ and $\bu_1,\dots,\bu_p$ the columns of $\tilde{\bU}$. For $j\in J$, we can write
\begin{equation}\label{eq:first_layer_lin_comb_master}
\bpsi(\bw_j) = \sum_{k=1}^r a_j^k\, \bpsi(\bw_{i_k}) \quad\text{for some } a_j^k\in\mathbb{R}
\end{equation}
If we define $\bU_1$ such that (denoting $\bu_{1,i}$ the $i$-th row of $\bU_1$)
\begin{equation*}
\bu_{1,i} = \bu_i + \sum_{k=1}^{n-r} a_k^i\,\bu_{j_k}   \quad\text{for }i \in I,\quad \bu_{1,j} = 0 \quad\text{for }j\in J
\end{equation*}
then $\bU_1\tilde{\bW} = \tilde{\bU}\tilde{\bW}$.
The path $t\in[0,1/2]\mapsto\btheta_t=(2t\,\bU_1+(1-2t)\tilde{\bU},\tilde{\bW})$ leaves the network unchanged, i.e. $\bPhi(\cdot;\tilde{\btheta}) = \bPhi(\cdot;\btheta_t)$ for $t\in[0,1/2]$. At this point, we can select $\bw_{1,j_1},\dots,\bw_{1,j_{p-r}}\in\mathbb{R}^n$ such that the matrix $\bW_1$ with rows $\bw_{1,i}=\bw_i$ for $i\in I$ and $\bw_{1,j}$ for $j\in J$, verifies $\rk(\bpsi(\bW_1)) = q$. Notice that the existence of such vectors $\bw_{1,j_k}$, $k\in [p-r]$, is guaranteed by the definition of $q = \mathrm{dim}^*(\sigma,\bX)$. The path $t\in[1/2,1]\mapsto\btheta_t=(\bU_1,(2t-1)\bW_1 + (2 -2t)\tilde{\bW})$ leaves the network unchanged, i.e. $\bPhi(\cdot;\btheta_0) = \bPhi(\cdot;\btheta_t)$ for $t\in[0,1]$. The new parameter value $\btheta_1 = (\bU_1,\bW_1)$ satisfies $\rk(\bpsi(\bW_1)) = q$.

\paragraph{Step 2.} By step 1, we can assume that $\rk(\tilde{\bW}) = q$. Since the network has the form $\bPhi(\bx;\btheta) = \bU\bpsi(\bW)\bvarphi(\bx)$ and since the function $\ell$ is convex, there exists $\bU^* \in \mathbb{R}^{m \times p}$ such that $\btheta = (\bU^*,\tilde{\bW}) \in \argmin_{\btheta} L(\btheta)$. The proof is therefore concluded by selecting the path $t \in [0,1] \mapsto \btheta_t = (t\bU^* + (1-t)\tilde{\bU},\tilde{\bW})$.

This shows that property \ref{p:conn} holds and therefore it proves the theorem.
\end{proof}

\subsection{Proof of Theorem \ref{theo:linear}}

The first step for proving Theorem \ref{theo:linear} consists in extending the result of Theorem \ref{master_theorem} to the case of one-hidden-layer linear NNs
$\bPhi(\bx;\btheta) = \bU\bW\bx$
with $\bU\in\mathbb{R}^{m\times p}$, $\bW\in\mathbb{R}^{p\times n}$ with $p<n$ and square loss functions $L(\btheta) = \expval\norm*{\bPhi(\bX;\btheta) -\bY}^2$. We start by pointing out a symmetry property of this type of networks:
for every $\bG \in GL(p)$ it holds that
\begin{equation}
\bPhi(\bx;(\bU,\bW)) = \bU\bW\bx = (\bU\bG^{-1})(\bG\bW)\bx = \bPhi(\bx;(\bU\bG^{-1},\bG\bW))
\end{equation}
This means that the map $\btheta\mapsto \bPhi(\cdot;\btheta)$ is defined up to an action of the group $GL(p)$ over the parameter space $\Theta = \mathbb{R}^{m\times p}\times \mathbb{R}^{p\times n}$; the same remark holds for the loss function $L(\btheta)$. We can therefore think about the loss function as defined over the topological quotient $\Theta / GL(p)$. We denote the orbit of an element $\btheta=(\bU,\bW)\in\Theta$ as
\begin{equation*}
[\btheta] = [\bU,\bW] = \bra{\bG\cdot\btheta = (\bU\bG^{-1},\bG\bW) \st \bG \in GL(p)}
\end{equation*}
If $g$ is a real-valued function defined on $\Theta$ such that $g(\bG\cdot \btheta) = g(\btheta)$ for all $\bG\in GL(p)$ and $\btheta\in\Theta$, then one can equivalently consider $g$ as defined on $\Theta/GL(p)$ as $g([\btheta]) = g(\btheta)$; for simplicity we denote $g[\btheta] = g([\btheta])$. This is exactly the case for the loss function $L(\btheta)$. 
In the proof of Theorem \ref{master_theorem}, we describe how to construct a path from an initial parameter value $\tilde{\btheta} = (\tilde{\bU},\tilde{\bW})$ to a parameter value $\btheta_1 = (\bq(\bW_1),\bW_1)$, with $\rk(\bW_1) = p$ and $\bq:\mathbb{R}^{p\times n} \to \mathbb{R}^{m\times p}$ the function defined by
$$
\bq(\bW) = \bSigma_{\bY\bX}\bW^T(\bW \bSigma_\bX \bW^T)^\dagger \in \argmin_\bU L(\btheta)|_{\btheta = (\bU,\bW)}
$$
(see Lemma \ref{lemma:best_solution_regression}). Therefore, let $\tilde{\btheta} = (\bq(\tilde{\bW}),\tilde{\bW})$ with $\rk(\tilde{\bW}) = p$, be an initial parameter. 
Since an optimal parameter is given by $\btheta = (\bq(\bW),\bW)$ for some $\bW$, we seek for a path in the form $\btheta_t = (\bq(\bW_t),\bW_t)$ with $\rk(\bW_t) = p$ for all $t\in[0,1]$. This path must be such that $t\mapsto L(\btheta_t)$ is non-increasing. If we assume that $\bSigma_\bX = \bI$, it holds
\begin{equation*}
L(\btheta_t) = \tr(\bSigma_\bY) - \tr(\bM\bP_{\bW_t})   
\end{equation*}
where $\bM$ is a PSD matrix and, for every matrix $\bW$, $\bP_{\bW}$ denotes the orthogonal projection on the rows of $\bW$, that is $\bP_{\bW} = \bW^\dagger\bW$  (see Lemma \ref{lemma:best_solution_regression}). Therefore it is equivalent for the path $\btheta_t =(\bq(\bW_t),\bW_t)$ to be such that the function
\begin{equation*}
t\in[0,1] \mapsto f(\bW_t) \doteq \tr(\bM\bP_{\bW_t})    
\end{equation*}
is non-decreasing. In particular, the function $f$ is defined up to the action of the group $GL(p)$ on $\Theta$. Since we look for $\bW_t$ of rank $p$, we can consider $f$ as defined on $G(p,n)$, the Grassmanian of $p$ dimensional linear subspaces of $\mathbb{R}^n$. The proof below for the linear one-hidden-layer case is articulated as follows. We first construct a path $[\bW_t]\in G(p,n)$ such that $[\bW_0] = [\tilde{\bW}]$, $[\bW_1]$ maximizes $f$ and such that the function $t\in[0,1]\mapsto f[\bW_t]$ is non-decreasing (Lemma \ref{lemma:grass}). We then show that such a path can be \emph{lifted} to a corresponding path $\bW_t\in \mathbb{R}^{p\times n}$ (Lemma \ref{lemma:lifting_the_path}). Finally, we show that we can drop the assumption $\bSigma_\bX = \bI$ and the result still holds (Lemma \ref{lemma:sigmaX_ass}).
\begin{lemma}\label{lemma:grass}
Let $[\tilde{\bW}]\in G(p,n)$ and assume $\bSigma_\bX = \bI$. Then there exists a continuous path $t\in[0,1]\mapsto[\bW_t]\in G(p,n)$ such that $[\bW_0] = [\tilde{\bW}]$, $[\bW_1]$ maximizes $f$ and such that the function $t\in[0,1]\mapsto f[\bW_t]$ is non-decreasing.
\end{lemma}
\begin{proof}
While it is geometrically intuitive that the results should hold, we derive a constructive proof. We start by noticing that if $[\bW] \in G(p,n)$ and $\bw_1,\dots,\bw_p$ is an orthonormal basis of $[\bW]$, then 
\begin{equation}\label{eq:loss_orth}
f[\bW] = \sum_{i=1}^p \bw_i^T\bM\bw_i
\end{equation}
Moreover, if $\bM = \sum_{j=1}^n \sigma_i \bv_j\bv_j^T$ is the SVD of $\bM$, where $\sigma_1\geq\dots\geq\sigma_n \geq 0$, then (\ref{eq:loss_orth}) can be written as 
\begin{equation*}
f[\bW] = \sum_{j=1}^n\sigma_j \sum_{i=1}^p \prodscal{\bv_j,\bw_i}^2
\end{equation*}
In particular the maximum of $f$ is obtained for $[\bW] = [\bV] \doteq [\bv_1,\dots,\bv_p]$ (with some abuse of notation, we identify a subspace with one of its basis). To prove the result is therefore sufficient to show a path $[\bW_t]$ from any $[\bW_0]=[\tilde{\bW}]$ to $[\bW_1]= [\bV]$, such that the function $t\in[0,1]\mapsto f[\bW_t]$ is non-decreasing. To do this we construct a finite sequence of paths
\begin{equation*}
[\bW_t^i] \quad\text{such that} \quad [\bW_0^i] = [\bW^{i-1}] \quad\text{and}\quad [\bW_1^i] = [\bW^i]   
\end{equation*}
for $i\in[p]$, with $[\bW^0] = [\tilde{\bW}]$, $[\bW^p] = [\bV]$ and
\begin{equation*}
\bW^i = [\bv_1,\dots,\bv_i,\bw_{i+1}^{i-1},\dots,\bw_p^{i-1}] \quad\text{for } i \in [p]
\end{equation*}
where $\bw_1^j =\bv_1,\dots,\bw_j^j=\bv_j,\bw_{j+1}^j,\dots,\bw_p^j$ is an orthonormal basis of $[\bW^j]$, for $j\in[0,p]$. Moreover, the paths $[\bW^i_t]$ are such that the functions $t\in[0,1]\mapsto f[\bW^i_t]$ are non-decreasing. Such paths are defined as follows. Let $i\in[0,p-1]$ and consider 
\begin{equation*}
[\bW^i]=[\bw^i_1=\bv_1,\dots,\bw^i_i=\bv_i,\bw^i_{i+1},\dots,\bw^i_p]
\end{equation*}
We define 
\begin{equation*}
\bu_{i+1}^i = \begin{cases}\frac{\bP_{\bW^i}\bv_{i+1}}{\norm{\bP_{\bW^i}\bv_{i+1}}} & \text{if } \bP_{\bW^i}\bv_{i+1} \neq \bzero \\  \bzero & \text{o.w.} \end{cases}    
\end{equation*}
Then we complete $\bv_1,\dots,\bv_i,\bu_{i+1}^i$ to an orthonormal basis of $[\bW^i]$: $$\bv_1,\dots,\bv_i,\bu_{i+1}^i,\dots,\bu_p^i$$ We call $\bw^{i+1}_j = \bu_j^i$ for $j\in[i+2,p]$ and we define
\begin{equation*}
[\bW^{i+1}] = [\bv_1,\dots,\bv_i,\bw_{i+1}^{i+1}=\bv_{i+1},\bw_{i+2}^{i+1},\dots,\bw_p^{i+1}]
\end{equation*}
The path $[\bW^i_t]$ is then obtained by moving $\bu_{i+1}^i$ to $\bv_{i+1}$ on a geodesic on the unit sphere $S^{n-1}\subset \mathbb{R}^n$, i.e.
\begin{equation*}
[\bW^{i+1}_t] = [\bv_1,\dots,\bv_i,\bu_{i+1}^i(t),\bu_{i+2}^i,\dots,\bu_p^i]
\end{equation*}
where we defined
\begin{equation*}
\bu_{i+1}^i(t) \,=\, (1 - (1 - \mu_{i+1})t)\bu_{i+1}^i \, + \, \sqrt{1-(1 - (1 - \mu_{i+1})t)^2}\cdot\frac{\bv_{i+1} - \mu_{i+1}\bu_{i+1}^i}{\sqrt{1-\mu_{i+1}^2}}
\end{equation*}
for $\mu_{i+1} = \prodscal{\bu_{i+1}^i,\bv_{i+1}}$.
The fact that the function $t\in[0,1] \mapsto f[\bW^{i+1}_t]$ is non-decreasing can be proved by noticing that
\begin{equation*}
f[\bW^{i+1}_t] - f[\bW^i] = \sum_{j=i+1}^n\sigma_j\prodscal{\bu_{i+1}^i(t),\bv_j}^2
\end{equation*}
and by showing that the derivative of the RHS is greater or equal than $0$.
This concludes the proof of the lemma.
\end{proof}
\begin{lemma}
\label{lemma:lifting_the_path}
Let $\tilde{\bW}\in \mathbb{R}^{p\times n}$ and assume $\bSigma_X = \bI$. Then there exists a continuous path $t\in[0,1]\mapsto \bW_t\in \mathbb{R}^{p\times n}$ such that $\bW_0 = \tilde{\bW}$, $\bW_1$ maximizes $f$ and such that the function $t\in[0,1]\mapsto f(\bW_t)$ is non-decreasing.
\end{lemma}
\begin{proof}
The only thing we need to prove in this case is that we can \emph{lift} the paths $[\bW^i_t]\in G(p,n)$ from the proof of Lemma \ref{lemma:grass} to continuous paths $\bW^i_t\in\mathbb{R}^{n\times p}$. We first notice that if the basis $\bra{\bw^i_1,\dots,\bw^i_p}$ and $\bra{\bw^i_1,\dots,\bw^i_i,\bu^i_{i+1},\dots,\bu^i_p}$ are defined as above, then we can assume (up to changing some signs) that they have all the same orientation, for all $i\in[0,p]$. Therefore we can define the matrices $\bW^i\in\mathbb{R}^{p\times n}$ with rows $\bw^i_1,\dots,\bw^i_p$ and the matrices $U^i\in\mathbb{R}^{p\times n}$ with rows $\bw^i_1,\dots,\bw_i^i,\bu_{i+1}^i,\dots,\bu^i_p$, for $i\in[0,p]$. The paths $\bW_t^{i+1}$ are defined in the same way as in the proof of Lemma \ref{lemma:grass}. Notice that such paths go from $\bW_0^{i+1} = \bU^i$ to $\bW_1^{i+1}=\bW^{i+1}$. It remains to construct paths from $\bW^i$ to $\bU^i$. Consider the matrix 
\begin{equation*}
\bO^i = \bW_i^T\bU_i \in SO(n)
\end{equation*}
Notice that $\bW^i\bO^i=\bU^i$. In particular there exist $\bA^i$ real skew-symmetric such that $\bO^i = e^{\bA^i}$.
Therefore the paths $t\in[0,1]\mapsto \bU^i_t = \bW^ie^{t\bA^i}$ go from $\bU^i_0=\bW^i$ to $\bU^i_1=\bU^i$. Moreover $f(\bU^i_t)$ is constant in $t$ (since the underlying linear subspace does not change). The only thing that remains to prove is that, given the matrix $\tilde{\bW}\in\mathbb{R}^{n\times p}$ with columns $\bw_1,\dots,\bw_p$, there is a path from $\tilde{\bW}$ to $\bW^0$. Now, $\bW^0$ was chosen as a matrix with orthonormal columns such that $[\tilde{\bW}]=[\bW^0]$. Therefore if $\tilde{\bW} = \bO\bLambda \bU$ is the SVD of $\tilde{\bW}$ with $\bU = \bW^0$, $\bLambda=\mathrm{diag}(\sigma_1,\dots,\sigma_p)\in\mathbb{R}^{p\times p}$ (with $\sigma_i>0$, $i\in [p]$) and $\bO\in SO(p)$, there exists $\bA$ real skew-symmetric such that $\bO=e^\bA$. Thus the path $t\in[0,1]\mapsto \bW_t = e^{(1-t)\bA}\bLambda^{1-t}\bW^0$ is a path between $\bW_0 = \tilde{\bW}$ and $\bW_1 = \bW^0$. This concludes the proof of the lemma. 
\end{proof}
\begin{lemma}\label{lemma:sigmaX_ass}
Lemma \ref{lemma:lifting_the_path} holds even if we drop the assumption $\bSigma_\bX = \bI$.
\end{lemma}
\begin{proof}
For sake of simplicity we distinguish two cases. \\
\emph{Case 1: $\rk(\bSigma_\bX) = n$.} 
Let $\bK = (\bSigma_\bX)^{1/2}$. Then $\tilde{\bX} = \bK^{-1}\bX$ is such that $\bSigma_{\tilde{\bX}} = \bI$. Therefore, if $t\in[0,1]\mapsto\btheta_t=(\bU_t,\bW_t)$ is the path given by Lemma \ref{lemma:lifting_the_path} for the case $\bX=\tilde{\bX}$, the sought path (for $\bX=\bX$) is given by $t\in[0,1]\mapsto (\bU_t,\bW_t\bK^{-1})$. \\
\emph{Case 2: $\rk(\bSigma_\bX) < n$.} 
In this case, if $r=\rk(\bSigma_\bX)$, $\bX$ belongs to a $r$-dimensional subspace of $\mathbb{R}^n$ (a.s.), call it $V$. If $\bO\in\mathbb{R}^{n\times r}$ is a matrix with an orthonormal basis of $V$ as columns, then $\bO\bO^T\bX = \bX$ (a.s.), and, if $\tilde{\bX}=\bO^T\bX$ then $\tilde{\bX}\in\mathbb{R}^r$ and $\rk(\bSigma_{\tilde{\bX}})=r$. Therefore, if $t\in[0,1]\mapsto\btheta_t=(\bU_t,\bW_t)$ is the path given by case 1 for $\bX=\tilde{\bX}$, the sought path (for $\bX=\bX$) is given by $t\in[0,1]\mapsto(\bU_t,\bW_t\bO^T)$.
\end{proof}
This concludes the proof of non-existence of spurious valleys for the square loss function of linear one-hidden-layer NNs $\bPhi(\bx;\btheta) = \bU\bW\bx$. The fact that such proof does not require any assumptions on the dimensions of the layers $n,p,m$ neither on the rank of the initial layers, allows us to prove non-existence of spurious valleys for the square loss function of linear NNs of any depth $K\geq 1$:
\begin{equation}\label{eq:multilayerLinearNN}
\bPhi(\bx;\btheta) = \bW_{K+1}\cdots \bW_1 \bx
\end{equation}
We start by proving a simple lemma. 
\begin{lemma}\label{lemma:matrix_fac}
Let $\tilde{\bU} = \tilde{\bM}^1\cdots \tilde{\bM}^n$, where $\tilde{\bU}\in\mathbb{R}^{r_0\times r_n}$ and $\tilde{\bM}^i\in\mathbb{R}^{r_{i-1}\times r_i}$. Suppose that $t\in[0,1]\mapsto \bU_t$ is a given continuous path between $\bU_0 = \tilde{\bU}$ and another matrix $\bU_1\in\mathbb{R}^{r_0\times r_n}$. If $r_i\geq \min\bra{r_0,r_n}$ for all $i$, then there exist continuous paths $\bM_t^i$ such that $\bM^i_0 = \tilde{\bM}^i$ and such that $\bU_t = \bM_t^1\dots \bM_t^n$. 
\end{lemma}
\begin{proof}
The statement can be proved by induction. If $n=1$ there is nothing to prove. Assume now (by induction) that it holds for all decompositions of $\bU_0$ with size less than $n$. Let $r = r_h = \min_{i\in[n-1]}r_i$ and assume (w.l.o.g.) that $r_n = \min\bra{r_0,r_n}$. We want to describe two paths $t\in[0,1]\mapsto \bV_t\in\mathbb{R}^{r_0\times r}$, $t\in[0,1]\mapsto \bW_t \in\mathbb{R}^{r\times r_n}$ such that $\bU_t = \bV_t\bW_t$ and $\bV_0 = \tilde{\bM}^1\cdots \tilde{\bM}^h$, $\bW_0 = \tilde{\bM}^{h+1}\cdots \tilde{\bM}^n$. By operating as in step 1 in the proof of Theorem \ref{master_theorem}, we can assume $\rk(\bW_0) = r_n$. Moreover (up to adding a linear path in $\bV_t$) we can assume that $\bV_0 = \bU_0\bW_0^\dagger$. We can then define $\bV_t = \bU_t\bW_0^\dagger$ and $\bW_t = \bW_0$ for $t\in(0,1]$. We thus factorized $\bU_t$ as $\bU_t = \bV_t\bW_t$. By induction, we can assume that we can factorize $\bV_t = \bM_t^1\cdots \bM_t^h$ and $\bW_t= \bM_t^{h+1}\cdots \bM_t^n$. This concludes the proof.
\end{proof}
We can now conclude the proof of Theorem \ref{theo:linear}.
\begin{proof}[Proof of Theorem \ref{theo:linear}]
Consider a linear network $\bPhi(\bx;\btheta)$ as in \eqref{eq:multilayerLinearNN}, where $$\bW_k \in \mathbb{R}^{p_k\times p_{k-1}} \quad\text{for } k \in [K+1]$$
We select $p_s =\min_{i\in[K]} p_k$. Then the network can be written as 
\begin{equation}\label{eq:linear_k_layers_composition}
\bPhi(\bx;\btheta) = \hat{\bW}^2 \hat{\bW}^1\,\bx \quad\text{where}\quad \hat{\bW}^2 = \bW^{K+1}\cdots \bW^{s+1},\quad \hat{\bW}^1 = \bW^s\cdots \bW^1
\end{equation}
Now we want to prove property that given an initial parameter $\tilde{\btheta} = (\tilde{\bW}^{K+1},\dots, \tilde{\bW}^1)$, there exists a continuous path $\btheta_t = (\bW^{K+1}_t,\dots, \bW^1_t)$ such that $L(\btheta_t)$ is non-increasing and such that $\btheta_0 = \tilde{\btheta}$ and $L(\btheta_1) = \min_\btheta L(\btheta)$. If we call $\hat{\tilde{\bW}}^i$, $i=1,2$, the matrices defined in \eqref{eq:linear_k_layers_composition} for $\btheta = \tilde{\btheta}$, then by Lemma \ref{lemma:sigmaX_ass} there exists a path $(\hat{\bW}_t^2,\hat{\bW}_t^1)$ satisfying the above. Thanks to Lemma \ref{lemma:matrix_fac}, we can decompose
\begin{equation}\label{eq:path_decompose}
\hat{\bW}_t^2 = \bW^{K+1}_t\cdots \bW^{s+1}_t, \quad  \hat{\bW}_t^1 = \bW^s_t\cdots \bW^1_t
\end{equation}
in a continuous way. Since $p_s$ was to chosen as the minimum, it also holds that
\begin{equation*}
\min_{\btheta = (\hat{\bW}^2,\hat{\bW}^1)} L (\btheta) = \min_{\btheta = (\bW^{K+1},\dots,\bW^1)} L(\btheta)
\end{equation*}
Therefore this is a suitable path and this concludes the proof of the theorem.
\end{proof}


\subsection{Proof of Theorem \ref{theo:quadratic_p_geq_n}}\label{sec:proof:quadratic}

\begin{proof}[Proof of Theorem \ref{theo:quadratic_p_geq_n}]
Let $\tilde{\btheta} = (\tilde{\bu},\tilde{\bW})$ be a starting parameter value.
We aim to construct a continuous path $t\in[0,1]\mapsto \btheta_t \in \Theta$ starting in $\btheta_0 = \tilde{\btheta}$ and such that $L(\btheta_1) = \min_{\btheta}L(\btheta)$ and such that the function $t\in[0,1]\mapsto L(\btheta_t)$ is non-increasing. Such a path can be constructed in two steps.
\paragraph{Step 1.} Let $\bA = \sum_{k=1}^p\tilde{u}_k\tilde{\bw}_k\tilde{\bw}_k^T$ and $\sum_{k=1}^n \bu_k^*\bw_k^*(\bw_k^*)^T$ be the SVD of $\bA$. We define the parameters value $\btheta^* = (\bu^*,\bW^*)$ where $\bu^* = (u_1^*,\dots,u_n^*,0,\dots,0)$ and $\bW^*$ is the $p\times n$ matrix with rows $\bw_i^*$ for $i \in [n]$ and $\bzero$ for $i \in [n+1,p]$. The first step consists in continuously mapping $\tilde{\btheta} = (\tilde{\bu},\tilde{\bW})$ to $\btheta^* = (\bu^*, \bW^*)$ with a path $\btheta_t$ such that $L(\btheta_t)$ is constant; the construction of such a path is detailed in Lemma \ref{movinglemma}.  

\paragraph{Step 2.} As noticed above, the network can be written as $\bPhi(\bx;\btheta) = \bu^T\sigma(\bW\bx) = \prodscal{\bA,\bM}_F$, where $\bA = \sum_{k=1}^p u_k\bw_k\bw_k^T$ and $\bM = \bx\bx^T$. The square loss $L(\btheta)$ is convex in the parameter $\bA$. Be $\bar{\bA}$ a minima of $L$ as function of $\bA$ and $\sum_{i=1}^n \bar{u}_k\bar{\bw}_k\bar{\bw}_k^T$ be the SVD of $\bar{\bA}$; also let $\bar{\bu}= (0,\dots,0,\bar{u}_1,\dots,\bar{u}_n)$ and $\bar{\bW}$ be the $p\times n$ matrix with rows $\bzero$ for $i\in[p-n]$ and $\bar{\bw}_i$ for $i\in[p-n+1,p]$. By the previous step we can assume that the initial parameter $\tilde{\btheta} = (\tilde{\bu},\tilde{\bW})$ is such that $\tilde{u}_i = 0$ and $\tilde{\bw}_i = \bzero$ for $i\in[n+1,p]$. Then the path $\btheta_t = (1-t)(\bu,\bW)+t(\bar{\bu},\bar{\bW})$ verifies property \ref{p:conn}. This indeed follows from the fact that $\bPhi(\bx;\btheta_t) = (1-t)\prodscal{\bA,\bM}_F + t\prodscal{\bar{\bA},\bM}_F$ and from the convexity of the loss $L$ as function of $\bA$.

This shows that property \ref{p:conn} holds and so it concludes the proof of Theorem \ref{theo:quadratic_p_geq_n}.
\end{proof}

To conclude the proof we just need to prove the following lemmas.
\begin{lemma}
\label{movinglemma}
Let $\btheta = (\bu, \bW)$ be an initial parameter and $\btheta^* = (\bu^*, \bW^*)$ be as in step 1 of the proof of Theorem \ref{theo:quadratic_p_geq_n}. Then there exists a continuous path $\btheta_t$ from $\btheta$ to $\btheta^*$ such that the loss $L(\btheta_t)$ is constant (as a function of $t$). 
\end{lemma}
\begin{proof}
Notice that we can assume $\bu \in \{-1, 0, 1\}^p$. This can be done simply scaling (continuously) each row $\bw_k$ of $\bW$ by $\sqrt{|u_k|}$. Assume first that $\bu \in \bra{\pm 1}^p$. The general case ($u_k=0$ for some $k$) is addressed in Remark \ref{remark:null-elements-quadratic}.
The sought path $\btheta_t$ can be constructed by iterating two steps (a finite amount of times). First we select a row $\bw_k$ and construct a continuous path that maps this row to one of the $\bw^*_i$; then we orthogonalize (w.r.t. such $\bw^*_i$) the rest of rows $\bw_j$, $j \neq k$. These two steps are constructed so that $\bA$ never changes and therefore the loss is constant. The first step is described in Lemma \ref{mappinglemma}, while the second is detailed in Lemma \ref{ortholemma}. At this point the parameter $\btheta=(\bu,\bW)$ verifies $u_i = u_i^*$, $\bw_i = \bw_i^*$ and $\bw_j \in \prodscal{\bra{\bw_i^*}}^\perp$ for $j\neq k$. In particular it holds
\begin{equation*}
\sum_{\substack{j=1 \\ j\neq i}}^n u_j^*\bw_j^*(\bw_j^*)^T = \sum_{\substack{j=1 \\ j\neq k}}^p u_k\bw_k\bw_k^T    
\end{equation*}
Therefore, an induction step applied on the reduced parameter values $$\bu_{-k} = (u_1,\dots,\widehat{u_k},\dots,u_p)$$ and $\bW_{-k} = [\bw_1,\dots,\widehat{\bw_k},\dots,\bw_p]^T P$, where $\bP = \sum_{j=1, j\neq i}^n \bw_j^*\be_j^T\in\mathbb{R}^{n\times (n-1)}$, concludes the proof. The fact that the non-zero components of $\bu$ and $\bW$ coincide with the first $n$ is not necessary, but we can clearly  assume it to hold w.l.o.g. 
\end{proof}
\begin{lemma}\label{mappinglemma}    
The first step described in the Proof of Lemma \ref{movinglemma} can be performed when $p > 2n$. 
\end{lemma}    
\begin{proof}
Let $E_+ = \bra{ k \in [p]\st u_k = 1}$, $E_- = \bra{ k \in [p]\st u_k = -1}$ and $p_+ = \abs{E_+}$, $p_- = \abs{E_-}$. Accordingly we define 
\begin{equation*}
\bW_+ = ([\bw_k]_{k\in E_+})^T\in\mathbb{R}^{p_+\times n} \quad \text{and}\quad \bW_- = ([\bw_k]_{k\in E_-})^T\in\mathbb{R}^{p_-\times n}
\end{equation*}
Notice that then we can write 
\begin{equation*}
\bA = \bW_{+}^T \bW_+ - \bW_{-}^T \bW_-
\end{equation*}
The main step of the proof is to observe that $\bA$ (and therefore the loss) is invariant to the action of orthogonal matrices $\bQ_{+} \in SO(p_{+})$ and $\bQ_{-} \in SO(p_{-})$. So, if $\bQ_{+}(t)$ (resp. $\bQ_{-}(t)$) is a continuous paths in $SO(p_{+})$ (resp. in $SO(p_{-})$) starting at the identity, acting on $\bW$ as 
\begin{equation*}
\bW_{+}(t) \doteq \bQ_{+}(t) \bW_{+}, \quad \bW_{-}(t) \doteq \bQ_{-}(t) \bW_{-}
\end{equation*}
we have that 
\begin{equation*}
\bA = \bW_{+}(t)^T \bW_{+}(t) - \bW_{-}(t)^T \bW_{-}(t)
\end{equation*}
is constant for all $t$. Now, since $p = p_{+} + p_{-} > 2n$, it follows that either $p_{+} > n$ or $p_{-} > n$. Assume w.l.o.g. that $p_{+} > n$. 
Since $p_{+}> n$, we can rotate the subspace generated by the columns of $\bW_{+}$ so that its first row is $\bzero$. 
That is, there exist $\bh \in \R^{p_{+}}$ non-zero such that
$\bh^T \bW_{+} = 0$ and $\|\bh \|=1$. It suffices to choose a path $\bQ(t)$ in $SO(p_+)$ whose first row equals $\bh$ at $t=1$. 
It follows that $\bQ(1) \bW_{+}$ has a first row equal to $\bzero$.
We then set the corresponding $u_1=0$, which does not change the loss, and finally set $\bw_1$ to the desired eigenvector $\bw_1^*$.
\end{proof}
\begin{lemma}
\label{ortholemma}
Assume that after the step in Lemma \ref{mappinglemma}, the first row of $\bW_+$ (resp. $\bW_-$) is given by $\bw_i^*$. Then we can map all the other rows of $\bW$ to be orthogonal to $\bw_i^*$, while keeping $\bA$ constant.
\end{lemma}
\begin{proof}
To simplify the notation we assume (w.l.o.g.) that $\bw_i^*=\bw_1^*$ and that 
\begin{equation*}
\bW = [\bw_1^*,\bw_2,\cdots,\bw_p]^T    
\end{equation*}
Now we want to construct a path
\begin{align*}
\bu_t & = (u_{1,t},u_2,\dots,u_p) \\
\bW_t & = [\bw_1^*,\bw_{2,t},\cdots,\bw_{p,t}]^T
\end{align*}
such that $\bw_{2,1},\dots,\bw_{p,1} \in \prodscal{\bra{\bw_1^*}}^\perp$. To do this we simply take
\begin{equation*}
\bw_{k,t} \doteq \bw_k - t\prodscal{\bw_1^*,\bw_k}\bw_1^*
\end{equation*}
If $\bA_t = \sum_{k=1}^p u_{k,t}\bw_{k,t}\bw_{k,t}^T$, we can show that there exists a choice of $u_{1,t}$ such that $\bA_t = \bA$ for all $t\in[0,1]$. It holds that
\begin{align*}
\bA_t & = u_{1,t}\,\bw_1^*(\bw_1^*)^T \\ & + \sum_{k=2}^pu_k\parq*{ (1-t)^2(w_k^1)^2\bw_1^*(\bw_1^*)^T +  (1-t)w_k^1\parr*{\tilde{\bw}_k(\bw_1^*)^T + \bw_1^*\tilde{\bw}_k^T} + \tilde{\bw}_k\tilde{\bw}_k^T}
\end{align*}
where $w_k^1 \doteq \prodscal{\bw_k,\bw_1^*}$ and $\tilde{\bw}_k = \bw_k - w_k^1\bw_1^*$. In particular
\begin{equation*}
\bA_t = \bV^*\left[
\begin{array}{c|c}
a_t & \bb_t^T \\
\hline
\bb_t & \bA_{2:n,2:n}
\end{array}
\right]
(\bV^*)^T
\end{equation*}
where $\bV^* = [\bw_1^*,\cdots,\bw_n^*] \in O(n)$. Since $\sum_{k=2}^p u_kw_k^1\,\tilde{\bw}_k = 0$, it follows
\begin{equation*}
\bb_t = (1-t)\sum_{k=2}^p u_kw_k^1\,\tilde{\bw}_k = 0 \quad \text{for all $t\in[0,1]$}
\end{equation*}
If we take 
\begin{equation*}
u_{1,t} = \lambda_1 -  (1-t)^2\sum_{k=2}^p u_k (w_k^1)^2
\end{equation*}
it holds that 
\begin{equation*}
a_t = u_{1,t} + (1-t)^2\sum_{k=2}^p u_k (w_k^1)^2 = \lambda_1 \quad\text{for all $t\in[0,1]$}
\end{equation*}
Therefore, $\bA_t = \bA$ constant. This concludes the proof of the lemma.
\end{proof}
\begin{remark}\label{remark:null-elements-quadratic}
In the proof of Lemma \ref{movinglemma}, we assumed that (after rescaling) $\bu\in\bra{\pm 1}^p$. In general, it could be that $u_k = 0$ for some $k$. In this case we can first map the corresponding vectors $\bw_k$ to $\bzero$ and the map such $u_k$ to $1$, without affecting the loss.
\end{remark}

\section{Proofs of Section \ref{sec:spurious valleys exist}}
\begin{proof}[Proof of Theorem \ref{theo:main_counter_ex}]
We consider here the case $m =1$, but the same proof can be extended to the case $m > 1$. We start by properly choosing a r.v. $(\bX,\bY)$. Be $\bar{\bX}\in\mathcal{R}_2(\sigma,n-1)$ a $(n-1)$ dimensional r.v. and $\bar{X}_n \in\mathcal{R}_2(\sigma,1)$ a one dimensional r.v. We consider $\tilde{\bX} = Z\bar{\bX}$, $X_n = (1-Z)\bar{X}_n$ and  $\bX = (\tilde{\bX},X_n)$, where $Z \sim \mathrm{Ber}(1/2)$ and $\bar{\bX},\bar{X}_n,Z$ are independent. By hypothesis, $p \leq 2^{-1}\mathrm{dim}_*(\sigma,\tilde{\bX})$. The proof is based on the fact that (for a proper choice of $\tilde{\bX}$) this implies that 
$\overline{V_{\sigma,p-1}^+}\neq V_{\sigma,p}^+$,
where we defined
\begin{equation*}
V_{\sigma,p}^+ = \bra{\Phi(\cdot;\btheta)\st \btheta \in [0,\infty)^{p}\times \mathbb{R}^{p\times n}} \subseteq L^2_{\bX}   
\end{equation*}
(see the remark at the end of the proof). The r.v. $Y$ is taken to be $Y = g_1(\bX) - g_2(\bX)$, where $g_2 = \beta\psi_{\sigma,\bv} \in V_{\sigma,1}^+$, $\beta > 0$, $\bv = \be_n$, and $g_1= \sum_{i=1}^p \alpha_i \psi_{\sigma,\bv_i} \in V_{\sigma,p}^+$,  $\balpha \in (0,\infty)^p$, $\bv_i \in \prodscal{\bra{\be_n}}^\perp$, $i \in [p]$, is such that 
\begin{equation*}
\inf_{f \in V_{\sigma,p-1}^+ }\expval\abs*{f(\bX) - g_1(\bX)}^2 = \epsilon > 0
\end{equation*}
We define 
\begin{align*}
V_{\sigma,(p-1,1)} = \bra*{f = f_1 - f_2 \st f_1 \in V_{\sigma,p-1}^+, f_2 \in V_{\sigma,1}^+}
\end{align*}
Notice that, for every path $\btheta:t \in [0,1]\mapsto \btheta_t \in \Theta$ such that $\Phi(\cdot;\btheta_0) \in V_{\sigma,p}^+$ and $\Phi(\cdot;\btheta_1) \in V_{\sigma,(p-1,1)}$, there exists $t_0 \in (0,1)$ such that $\Phi(\cdot;\btheta_{t_0}) \in V_{\sigma,p-1}^+$. Consider the \emph{lifted} square loss function $L : V_{\sigma,p} \to [0,\infty)$  defined as
\begin{equation*}
L(f) = \expval\abs*{f(\bX) - g(\bX)}^2   \qquad \text{for } f \in V_{\sigma,p}
\end{equation*}
We want to show that
\begin{align*}
L_{(p-1,0)} & \doteq \min_{f\in V_{\sigma,p-1}^+ }L(f) > L_{(p,0)} \doteq \min_{f\in V_{\sigma,p}^+}L(f) > L_{(p-1,1)} \doteq \min_{f\in V_{\sigma,(p-1,1)}}L(f)
\end{align*}
It holds that 
\begin{align*}
L_{(p-1,0)} & = \min_{f\in V_{\sigma,p-1}^+ } \bra*{ \expval\abs*{f(\bX) - g_1(\bX)}^2 } + 2 \min_{f\in V_{\sigma,p-1}^+ } \bra*{ \expval\parq*{ f(\bX) g_2(\bX) }} \\ & + \expval\abs*{g_2(\bX)}^2 - C\sigma(0) \\
& \geq \epsilon + L_{(p,0)} - C\sigma(0)
\end{align*}
where $C = \expval\parq{ g_1(\tilde{\bX}) } +  \expval\parq{ g_2(X_n) } $ ,and that
\begin{align*}
L_{(p,0)} & = \min_{f\in V_{\sigma,p}^+ } \bra*{ \expval\abs*{f(\bX) - g_1(\bX)}^2 } + 2 \min_{f\in V_{\sigma,p}^+ } \bra*{ \expval\parq*{ f(\bX) g_2(\bX) }} \\ & + \expval\abs*{g_2(\bX)}^2   - C\sigma(0) \\
& \geq \beta^2 \expval\abs*{\psi_{\sigma,\bv}(X_n)}^2 - C\sigma(0)
\end{align*}
Finally, it holds that
\begin{equation*}
L_{(p-1,1)} \leq \min_{i\in[1,p]} \alpha_i^2 \expval\abs{\psi_{\sigma,\bv_i}(\bX)}^2   
\end{equation*}
Given $M > 0$, up to multiply $g_1$ by a positive constant, it holds that
\begin{align*}
\epsilon & \geq M + C\sigma(0) \\    
\beta^2 & \geq \frac{M + C\sigma(0) + \min_{i\in[1,p]} \alpha_i^2 \expval\abs{\psi_{\sigma,\bv_i}(\bX)}^2}{\expval\abs{\psi_{\sigma,\bv}(X_n)}^2}
\end{align*}
To finish the proof, consider $\mathcal{U} = \bra{ \btheta = (\bu,\bW) \in \Theta \st \bu \in (0,\infty)^p }$ and $\btheta^* \in \mathcal{U}$ such that
$$
L(\btheta^*) = \min_{\btheta \in \mathcal{U}} L(\btheta)
$$
Then, (by continuity of $L$) there exists a neighborhood $\btheta^* \in \Omega \subset \mathcal{U}$ such that $\sup_{\btheta\in\Omega} L(\btheta) \leq L(\btheta^*) + M/2$. The set $\Omega$ then verifies the statement of the theorem.

\end{proof}

\begin{remark}
In the proof of Theorem \ref{theo:main_counter_ex} we used the fact that, if $p\geq 1$ verifies $p\leq \frac{1}{2}\dim_*(\sigma,n)$, then there exist $\bX \in \mathcal{R}_2(\sigma,n)$ such that 
$\overline{V_{\sigma,p-1}^+} \neq V_{\sigma,p}^+$
(in the $L^2_\bX$ metric). Assume $\bX$ is a $n$-dimensional standard Gaussian variable. 
Consider first the case $\sigma(z) = z^k$.
If $k=2$, then
\begin{equation*}
V_{\sigma,p}^+ \simeq \bra{ \bM \in \mathrm{S}^2(\mathbb{R}^n) \st \bM \text{ is PSD}}    
\end{equation*}
In particular, this implies that $V_{\sigma,p-1}^+$ is not dense in $V_{\sigma,p}^+$ if $p\leq n = \dim_*(\sigma,n)$ (which justifies the statement of Corollary \ref{cor: spurious polynomial}). If $k>2$, let $p\leq \frac{1}{2}\dim_*(\sigma,n)$ and assume that 
$\overline{V_{\sigma,p-1}^+} = V_{\sigma,p}^+$.
This implies that every tensor $\bT = 
\sum_{i=1}^p \bv_i^{\otimes k}$ can be 
approximated up to any accuracy by 
$\bT = \sum_{i=1}^{p-1} \tilde{\bv}_i^{\otimes k}$ for some $\tilde{\bv}_1,\dots,\tilde{\bv}_{p-1}\in\mathbb{R}^n$. But this also implies that every tensor $\bT \in \mathbb{S}^{k}(\mathbb{R}^n)$ has 
border %
rank 
less or equal than $2\parr*{\frac{1}{2}\dim_*(\sigma,n) -1} = \rk_\mathrm{S}(k,n) - 2$, 
which contradicts the definition of $\rk_\mathrm{S}(k,n)$.
For non-polynomial $\sigma$, we can get the same result, by using the decomposition \eqref{eq:norm of phi} and proceeding as above. 
\end{remark}

\section{Proof of Theorem \ref{theo:energy_gap}}

\begin{proof}
If we denote by $d\mu$ the probability distribution of $\bX$, the continuous function 
\begin{equation*}
\psi : (\bw, \bx) \in \mathbb{S}^n \times \mathbb{R}^n \mapsto \psi_\bw(\bx) = \sigma(\prodscal{\bw,\bx})     
\end{equation*}
belongs to $L^2(\mathbb{S}^n\times \mathbb{R}^n, d\tau \otimes d\mu)$. We consider the kernel associated with the neural network architecture
\begin{equation}\label{eq:kernel}
k(\bx,\by) = \int_\mathcal{W} \psi_\bw(\bx) \psi_\bw(\by)\,d\tau(\bw)
\end{equation}
The above defines a continuous symmetric, positive semi-definite kernel $k$, along with $\mathbb{H}$, the RKHS associated, and the integral operator $\Sigma : L^2(\mathbb{R}^n, d\mu) \to \mathbb{H} \subseteq L^2(\mathbb{R}^n, d\mu)$ defined as
\begin{equation*}
f \mapsto \parr*{\Sigma f : \bx \mapsto \int_{\mathbb{R}^n} f(\by) k(\bx,\by)\,d\mu(\by)}
\end{equation*} 
The operator $\Sigma$ admits a spectral decomposition in $L^2(\mathbb{R}^n, d\mu)$: $\Sigma e_k = \lambda_k e_k$ for an orthonormal basis $\bra*{e_k}_{k\geq 1}$ of $L^2(\mathbb{R}^n, d\mu)$ and non-increasing sequence of 
non-negative eigenvalues $\bra{\lambda_k}_{k\geq 1}$. Moreover the RKHS $\mathbb{H}$ is dense in $L^2(\mathbb{R}^n,d\mu)$ (see Lemma \ref{lemma:rkhs_dense}), 
which is equivalent to have $\lambda_k > 0$ for all $k\geq 1$. The expectation in \eqref{eq:kernel} provides a singular value decomposition for $\Sigma$ 
in terms of functions in $L^2(\mathbb{S}^n, d\tau)$. Indeed, given $g \in L^2(\mathbb{S}^n, d\tau)$, the linear operator 
$T: L^2(\mathbb{S}^n, d\tau) \to L^2(\mathbb{R}^n, d\mu)$ defined as
\begin{equation*}
g \mapsto \parr*{T g : \bx \mapsto \int_{\mathbb{S}^n} g(\bw) \psi_\bw(\bx)\,d\tau(\bw)}
\end{equation*}
satisfies $\Sigma = T T^*$. It follows that there exists an orthonormal basis 
of $L^2(\mathbb{S}^n,d\tau)$, $\bra{f_k}_{k\geq 1}$ such that $T f_k = \lambda_k^{1/2} e_k$ and therefore
$\psi_\bw = \sum_{k=1}^\infty \lambda_k^{1/2} f_k(\bw) e_k$.
Finally, it can be shown \citep{bach2017breaking} that in fact $\mathbb{H} = \mathrm{Im}(T)$, and 
thus $\mathbb{H}$ consists of functions $f$ that can be written, for some $g\in L^2(\mathbb{S}^n, d\tau)$ as
\begin{equation*}
f(\bx) = \int_{\mathbb{S}^n} g(\bw) \psi_\bw(\bx) d\tau(\bw) = \prodscal{g, \psi(\cdot,\bx)}_{L^2(\mathbb{S}^n, d\tau)} \quad \text{for } \bx \in \mathbb{R}^n
\end{equation*}
For an account of these properties, we refer to Bach \citep{bach2017equivalence}. Thanks to the density of $\mathbb{H}$ in $L^2(\mathbb{R}^n,d\mu)$, we can assume, without loss of generality, that
\begin{equation*}
f^*(\bx) = \int_{\mathbb{S}^n} g^*(\bw) \psi_\bw(\bx) d\tau(\bw)
\end{equation*}
for some $g^* \in L^2(\mathbb{S}^n,d\tau)$.
Now, given an initial set of first layer weights $\bw_1, \dots, \bw_p \in \mathbb{S}^n$ 
sampled i.i.d. from $d\tau$, and $\bW = [\bw_1,\dots,\bw_p]^T$, we define the empirical kernel 
\begin{equation*}
k_\bW(\bx,\by) = \frac{1}{p} \sum_{i=1}^p \sigma(\prodscal{\bx, \bw_i}) \sigma(\prodscal{\by, \bw_i})
\end{equation*}
which in turn defines an empirical RKHS $\mathbb{H}_\bW$. Keeping the first layer weights fixed and optimizing the output layer weights thus gives us 
the ability to find a function $f^*_\bW \in \mathbb{H}_\bW$ that best approximates $f^*$:
\begin{equation*}
\norm*{ f^*_\bW - f^* }_{L^2(\mathbb{R}^n,d\mu)}  = \min_{f \in \mathbb{H}_\bW}\norm*{f - f^*}_{L^2(\mathbb{R}^n,d\mu)} \doteq R(\bW)
\end{equation*}
Given an initial parameter parameter value $\tilde{\btheta} = (\tilde{\bu},\tilde{\bW})$ (here we incorporated $\tilde{\bb}$ in $\tilde{\bW}$) as in the statement, consider the path
\begin{equation*}
\btheta_t = (t\bq(\tilde{\bW}) + (1-t)\tilde{\bu},\tilde{\bW}) \quad\text{where}\quad \bq(\tilde{\bW}) = \argmin_{\bu\in\mathbb{R}^p} L(\btheta)|_{\btheta = (\bu,\tilde{\bW})}
\end{equation*}
By convexity of $L$, the function $t\in[0,1]\mapsto L(\btheta_t)$ is non-increasing and it holds that
\begin{align*}
L(\btheta_1) & \leq \mathcal{R}\parr{\bX,Y} + R(\tilde{\bW}) 
\end{align*}
Applying Proposition 1 from Bach \citep{bach2017equivalence}, it holds that
\begin{equation*}
R(\bW) \leq 4\lambda \quad\text{if}\quad p \geq 5d(\lambda)\log(16d(\lambda) / \delta)   
\end{equation*}
with probability greater or equal than $1-\delta$, where 
\begin{align*}
d(\lambda) &= \max_{\bw \in \mathbb{S}^n}\expval\parq*{\varphi_\bw(\bX)((\Sigma + \lambda I)^{-1}\psi_\bw)(\bX) } \\
&= \max_{\bw \in \mathbb{S}^n} \sum_{k=1}^\infty \frac{\lambda_k}{\lambda_k + \lambda} f_k(\bw)^2 \leq \lambda^{-1}\max_{\bw \in \mathbb{S}^n} \sum_{k=1}^\infty \lambda_k f_k(\bw)^2 = \lambda^{-1}\max_{\bw \in \mathbb{S}^n}\norm{\psi_\bw}^2_{L^2(\mathbb{R}^n,d\mu)}
\end{align*}
This shows part 1 of the statement. To prove part 2, notice that $f^*_\bW = \Phi(\cdot;\btheta)$ with $\btheta = (\bu^*_\bW, \bW)$ for some $\bu^*_\bW \in \mathbb{R}^p$.
By taking $u_k = \frac{1}{p} g^*(\bw_k)$ for $k \in [p]$ and denoting $Z(\bw):= g^*(\bw) \psi_\bw$ and by $Z$ the r.v. $Z = Z(\bv)$, for $\bv \sim d\tau$, with values in $L^2(\mathbb{R}^n,d\mu)$, it holds
\begin{equation}\label{eq:whp_error_bound}
R(\bW) \leq \norm[\Big]{\frac{1}{p}\sum_{k=1}^p Z(\bw_k)- \E_\tau\parq*{Z} }_{L^2(\mathbb{R}^n, d\mu)}
\end{equation}
Note that $C \doteq \sup_{\bw \in \mathbb{S}^n} \norm{Z(\bw)}_{L^2(\mathbb{R}^n,d\mu)} \leq \norm{g^*}_{L^\infty(\mathbb{S}^n,d\tau)} \max_{\bw \in \mathbb{S}^n} \norm{\psi_\bw}_{L^2(\mathbb{R}^n,d\mu)}< \infty$ if $\norm{g^*}_{L^\infty(\mathbb{S}^n,d\tau)} < \infty$. Then, applying Lemma \ref{lemma:hoeffding_hilbert} to the bound \eqref{eq:whp_error_bound}, we get that
\begin{equation*}
\mathbb{P}_\tau \bra*{ R(\bW) \leq \varepsilon } \geq 1 - \mathrm{exp}\bra*{-\parr*{\varepsilon - \sqrt{v(p)}}^2 / (2 v(p))}
\end{equation*}
for every $\varepsilon \geq \sqrt{v(p)}$, with $v(p) = C^2 /p$. The result follows by taking $\varepsilon = v(p)^{1/2} + p^{\delta/2-1/2}$.
\end{proof}
\section{Proofs of Section \ref{sec:dimension}}\label{app:proofs:id}

\begin{proof}[Proof of Lemma \ref{lemma:uid_when_infty}]
If $\sigma$ is a polynomial of any degree $d$, then it holds that $\dim^*(\sigma,n) < \infty$. 
Indeed, let $\sigma(z) = a_0 + a_1z + \cdots + a_dz^d$, for some $a_k \in \mathbb{R}$. If $I = \bra{k \in [0,d] \st a_k \neq 0}$, then
\begin{equation*}
V_\sigma \subseteq \mathbb{R}_I[\bx] \doteq \bra{\bx\mapsto \sum_{k \in I} \sum_{\abs{\bbeta} = k} \alpha_\bbeta \bx^\bbeta \st \alpha_\bbeta \in\mathbb{R}}  
\end{equation*}
It follows that
\begin{equation*}
\dim^*(\sigma,n) = \dim(V_\sigma) \leq \dim\parr*{\mathbb{R}_I[\bx]} = \sum_{k=0}^d\binom{n + k -1}{k}\mathbf{1}_{\bra{a_k\neq 0}} = O(n^d)
\end{equation*}
This proves one implication. We prove the other one by contradiction. Assume now that $\sigma$ is not a polynomial and that $\dim(V_\sigma) = q < \infty$. Thanks to Theorem 1 in Leshno et al. \citep{leshno1993multilayer}, for every continuous function $g:\mathbb{R}^n\to\mathbb{R}$, any compact set $K\subset \mathbb{R}^n$, and any $\varepsilon >0$ there exist $h \in V_\sigma$ such that 
\begin{equation}\label{eq:continuous_universal_approx}
\sup_{\bx\in K}\,\abs{h(\bx) - g(\bx)} < \varepsilon    
\end{equation}
Now, let $g:\mathbb{R}^n\to\mathbb{R}$ be a continuous function supported on a compact set $C\subset \mathbb{R}^n$. We call $C_c(\mathbb{R}^n)$ the set of the real-valued continuous functions from $\mathbb{R}^n$ with compact support.
Thanks to \eqref{eq:continuous_universal_approx}, we can find a sequence of compact sets $\bra{K_m}_{m\geq 1}$ of $\mathbb{R}^n$ such that
\begin{equation*}
C \subseteq K_1 \subseteq K_2 \subseteq \cdots \subseteq K_m \subseteq \cdots \subseteq \cup_{m=1}^\infty K_m = \mathbb{R}^n
\end{equation*}
and a sequence of functions $\bra{h_m}_{m\geq 1} \subset V_\sigma$ such that
\begin{equation*}
\norm{g - h_m\mathbbm{1}_{K_m}}_{L^2_\bX} = \norm{(g - h_m)\mathbbm{1}_{K_m}}_{L^2_\bX} < 2^{-m} 
\end{equation*}
In particular this implies that 
\begin{equation*}
\norm{h_n\mathbbm{1}_{K_n} - h_m\mathbbm{1}_{K_m}}_{L^2_\bX} < 2^{1-\min\bra{n,m}}\to 0
\end{equation*}
as $n,m \to \infty$, i.e. $\bra{h_m\mathbbm{1}_{K_m}}_{m\geq 1}$ is a Cauchy sequence in $L^2_\bX$ and therefore it admits a limit $\lim_{n\to\infty} h_m\mathbbm{1}_{K_m} = g\in L^2_\bX$. Since $\dim(V_\sigma) = q < \infty$, there exists $\bw_1,\dots,\bw_q \in \mathbb{R}^n$ such that every $h\in V_\sigma$ can be written as 
\begin{equation*}
h(\bx) = \prodscal{\bu, \bgamma(\bx)}
\end{equation*}
for some $\bu \in \mathbb{R}^q$, where $\bgamma(\bx) = (\sigma(\prodscal{\bw_1,\bx}),\dots,\sigma(\prodscal{\bw_q,\bx}))$. Let $\bra{\bu_m}_{m\geq 1}\subset\mathbb{R}^q$ such that $h_m(\bx) = \prodscal{\bu_m, \bgamma(\bx)}$. Thanks to the above calculations, we know that the sequence $\bra{\norm{ h_m\mathbbm{1}_K }_{L^2_\bX}}_{m\geq 1}$ is bounded for any arbitrary compact set $K\subseteq \mathbb{R}^n$. Since
\begin{equation*}
\norm{ h_m\mathbbm{1}_K }_{L^2_\bX}^2 = \bu_m^T M \bu_m
\end{equation*}
where $M = \expval\parq*{\bgamma(\bX)\bgamma(\bX)^T\mathbbm{1}_{\bra{\bX \in K}}} \in \mathbb{R}^{q\times q}$, this implies that the sequence $\bra{\bu_m}_{m\geq 1}$ is bounded (unless $g=0$). Therefore (up to extracting a sub-sequence) we can assume that it has a limit $\bu\in\mathbb{R}^q$. If we call $h\in V_\sigma$ the function defined as $h(\bx) = \prodscal{\bu, \bgamma(\bx)}$, it is easy to check (from the above calculations) that $h = g$ in $L^2_\bX$. This shows that $C_c(\mathbb{R}^n) \subseteq V_\sigma$, which in turn implies that $V_\sigma$ is dense in $L^2_\bX$ (since $C_c(\mathbb{R}^n)$ is dense in $L^2_\bX$). But this is impossible, since $\dim(V_\sigma) = q < \infty = \dim(L^2_\bX)$. Therefore, it must hold $\dim(V_\sigma) = \infty$.
\end{proof}

\begin{proof}[Proof of Lemma \ref{lemma:lid}]
Let $\btheta = (\bu,\bW) \in [0,\infty)^p\times \mathbb{R}^{p\times n}$. For every $\bx \in \mathbb{R}^n$ it holds 
\begin{equation*}
\Phi(\bx; \btheta) = \sum_{i=1}^p u_i (\prodscal{\bw_i,\bx })^k = \sum_{i=1}^p u_i \prodscal{\bw_i^{\otimes k},\bx^{\otimes k}}_F = \prodscal[\Big]{\sum_{i=1}^p u_i\bw_i^{\otimes k},\bx^{\otimes k}}_F
\end{equation*}
For any $p \geq 1$ and $(\bu,\bW) \in [0,\infty)^p\times \mathbb{R}^{p\times n}$, $\sum_{i=1}^p u_i\bw_i^{\otimes k} \in S^k(\mathbb{R}^n)$. By definition of $\rk_\mathrm{S}(k,n)$, it follows that there exists $q \leq \rk_\mathrm{S}(k,n)$ and $\tilde{\btheta} = (\tilde{\bu},\tilde{\bW}) \in [0,\infty)^q\times \mathbb{R}^{q\times n}$ such that 
\begin{equation*}
\sum_{i=1}^p u_i\bw_i^{\otimes k} = \sum_{i=1}^q \tilde{u}_i\tilde{\bw}_i^{\otimes k} \quad\Rightarrow\quad \Phi(\cdot; \btheta) = \Phi(\cdot; \tilde{\btheta})
\end{equation*}
By definition of $\dim_*(\sigma,n)$, this implies that $\mathrm{dim}_*(\sigma,n) \leq \mathrm{rk}_\mathrm{S}(k,n)$. The equality follows by choosing $(\bu,\bW) \in [0,\infty)^p\times \mathbb{R}^{p\times p}$ such that $\mathrm{rk}_\mathrm{S}\parr{\sum_{i=1}^p u_i\bw_i^{\otimes k}} = \mathrm{rk}_\mathrm{S}(k,n).$
\end{proof}


\begin{proof}[Proof of Lemma \ref{lemma:infinite lower dimension}]
If $\sigma$ is polynomial, then one implication follows by Lemma \ref{lemma:uid_when_infty}. Now, assume that $\sigma \in L^2(\mathbb{R}, e^{-x^2/2}\,dx)$ is a continuous non-polynomial activation and let $\bX \sim N(\bzero,\bI)$ be a r.v. in $\mathcal{R}_2(\sigma,n)$.
Then, we can write $\sigma(z) = \sum_{k=0}^\infty \hat{\sigma}_k h_k(z)$, where $h_k$ is the $k$-th Hermite polynomial. It follows that, for $\btheta = (\bu,\bW)$,
\begin{equation}\label{eq:norm of phi}
\expval\abs{\Phi(\bX;\btheta)}^2 = \sum_{k=1}^\infty \hat{\sigma}_k^2 \norm*{\sum_{i=1}^p u_i \bw_i^{\otimes k}}_F^2 
\end{equation}
(see Lemma 1 from \citep{mondelli2018connection}). Since $\sigma$ is not polynomial and $n > 1$, $V_{\sigma,p} \neq V_{\sigma,p+1}$, where $$V_{\sigma,p} \doteq \bra*{ \bx \mapsto \sum_{k=1}^p u_k\sigma(\prodscal{\bw_k, \bx}) \st (\bu,\bW) \in \R^{p}\times\R^{p \times n}}$$ 
Indeed, if $V_{\sigma,p} = V_{\sigma,p+1}$, then $V_{\sigma,p} = V_{\sigma,q}$ for every $q > p$. Let $k$ be a positive integer such that $\hat{\sigma}_k \neq 0$ and such that $\rk_\mathrm{S}(k,n) > p$. 
Let $\bG = \sum_{i=1}^{q} \alpha_i\bv_i^{\otimes k}$ a symmetric tensor with $\rk_\mathrm{S}(\bG) = q = \rk_\mathrm{S}(k,n)$ and $g = \sum_{i=1}^q \alpha_i\psi_{\sigma,\bv_i}$. If $g \in V_{\sigma,p}$, then there exists $\btheta = (\bu,\bW) \in \Theta = \mathbb{R}^p \times \mathbb{R}^{p\times n}$ such that
\begin{equation*}
0 = \expval\abs{\Phi(\bX;\btheta) - g(\bX)}^2 = \sum_{k=1}^\infty \hat{\sigma}_k^2 \norm[\bigg]{\sum_{i=1}^p u_i \bw_i^{\otimes k} - \sum_{i=1}^{q} \alpha_i\bv_i^{\otimes k}}_F^2
\end{equation*}
But this would imply that $\rk_\mathrm{S}(\bG) \leq p$, which is a contradiction. This concludes the proof.
\end{proof}

\section{Proofs of Additional Lemmas}
\label{app:lemmas}

\begingroup
\def\thetheorem{\ref{lemma:what_are_spurious_valleys}}
\begin{lemma}
Be $\btheta \mapsto L(\btheta)$ a continuous function. Then, property \ref{p:conn} implies absence of spurious valleys. In particular, this implies absence of strict spurious minima, and of (generally non-strict) spurious minima if property \ref{p:conn} holds with strictly decreasing paths $t\mapsto L(\btheta_t)$. Conversely, presence of spurious valleys implies existence of spurious minima.
\end{lemma}
\addtocounter{theorem}{-1}
\endgroup

\begin{proof}
Assume that property \ref{p:conn} holds. Consider any value $c>0$ such that $\Omega_L(c)$ is non-empty and let $\mathcal{U}$ be a connected component of $\Omega_L(c)$. Given a point $\btheta\in\mathcal{U}$ there exists a path from $\btheta$ satisfying property \ref{p:conn}. This means that $\mathcal{U}$ contains a global minima, and therefore it can not be a spurious valley.
Similarly, assume that property \ref{p:conn} holds with strictly decreasing paths and that the function $L$ admits a strict local minima. This means that there exists a point $\btheta_0$ such that $\min_\btheta L(\btheta) < L(\btheta_0) < L(\btheta)$ for all $\btheta$ in $B_\epsilon(\btheta)$, for some $\epsilon > 0$. But this implies that for any path $t\in[0,1]\mapsto \btheta_t$ if holds $L(\btheta_t) > L(\btheta_0)$ for some $t>0$ sufficiently small, a contradiction. To see the last point, assume that there exist spurious valleys and consider $\mathcal{U}$ a connected component of $\Omega_L(c)$ for some $c>0$. Then $\btheta^* \in \argmin_\btheta L(\btheta)$ is a spurious minima. 
\end{proof}

\begin{lemma}\label{lemma:best_solution_regression}
Consider the optimization problem
\begin{equation}\label{eq:regression_problem}
\argmin_{\bW\in\mathbb{R}^{m\times n}} \ell(\bW) \quad \text{where}\quad \ell(\bW) = \expval\norm{\bW\bX-\bY}^2 
\end{equation}
for two square integrable r.v.'s $\bX$ and $\bY$ with values in $\mathbb{R}^n$ and $\mathbb{R}^m$ respectively. Then one solution to (\ref{eq:regression_problem}) is given by
\begin{equation}\label{eq:regression_solution}
\bW = \bSigma_{\bY\bX} \bSigma_\bX^\dagger
\end{equation}
Similarly, one solution to the optimization problem 
\begin{equation}\notag
\argmin_{\bU\in\mathbb{R}^{m\times p}} \ell(\bU;\bW) \quad \text{where}\quad \ell(\bU;\bW) = \expval\norm{\bU\bW\bX-\bY}^2
\end{equation}
for any $\bW\in\mathbb{R}^{p\times n}$ is given by
\begin{equation}\label{eq:q_function}
\bU = \bq(\bW) \doteq \bSigma_{\bY\bX} \bW^T(\bW\bSigma_\bX\bW^T)^\dagger
\end{equation}
Assuming $\bSigma_\bX$ invertible, the minimal value obtained by $\ell(\bU;\bW)$ is given by 
\begin{equation}\label{minima_regression_value_second_layer}
\ell(\bq(\bW);\bW) = \tr(\bSigma_\bY) - \tr((\bW\bK)^\dagger (\bW\bK) \bM)
\end{equation}
where $\bK = (\bSigma_\bX)^{1/2}$ and $\bM = \bK^{-1}\bSigma_{\bX\bY}\bSigma_{\bY\bX}\bK^{-1}$. If $\bM = \sum_{i=1}^n \lambda_i \bv_i \bv_i^T$ is the SVD of $\bM$, the quantity \eqref{minima_regression_value_second_layer} is minimized over $\bW$ for $(\bW\bK)^\dagger (\bW\bK) = \sum_{i=1}^{p\wedge n} \bv_i\bv_i^T$.
\end{lemma}
\begin{proof}
The first part of the lemma can be shown by writing problem (\ref{eq:regression_problem}) as 
\begin{equation}\label{eq:regression_function}
\argmin_{\bW\in\mathbb{R}^{m\times n}} \ell(\bW) \quad \text{where}\quad \ell(\bW) = \tr(\bW\bSigma_\bX\bW^T) - 2\tr (\bSigma_{\bY\bX}\bW^T) 
\end{equation}
and by taking $\bW$ as a stationary point of the above $\ell(\bW)$. Using this fact, one minima of the function $\ell(\bU;\bW)$ is given by 
\begin{equation*}
\bU = \bSigma_{\bY\bX\bW} (\bSigma_\bW\bX)^\dagger = \bSigma_{\bY\bX} \bW^T(\bW\bSigma_\bX\bW^T)^\dagger
\end{equation*}
Now assume that $\bSigma_\bX$ is invertible; let $\bK = (\bSigma_\bX)^{1/2}$ and $\bM = \bK^{-1}\bSigma_{\bX\bY}\bSigma_{\bY\bX}\bK^{-1}$. Then it holds
\begin{align*}
\ell(\bq(\bW);\bW) & =  \tr(\bq(\bW)\bW\bSigma_\bX\bW^T\bq(\bW)^T) - 2\tr (\bSigma_{\bY\bX}\bW^T\bq(\bW)^T) + \tr(\bSigma_\bY) \\
& = \tr(\bSigma_{\bY\bX} \bW^T(\bW\bSigma_\bX\bW^T)^\dagger \bW\bSigma_\bX\bW^T(\bW\bSigma_\bX\bW^T)^\dagger \bW\bSigma_{\bX\bY}) \\ 
& - 2\tr (\bSigma_{\bY\bX}\bW^T(\bW\bSigma_\bX\bW^T)^\dagger \bW\bSigma_{\bX\bY}) + \tr(\bSigma_\bY) \\
& = -\tr (\bSigma_{\bY\bX}\bW^T(\bW\bSigma_\bX\bW^T)^\dagger \bW\bSigma_{\bX\bY}) + \tr(\bSigma_\bY) \\
& = -\tr(\bM(\bW\bK)^T((\bW\bK)(\bW\bK)^T)^\dagger (\bW\bK)) + \tr(\bSigma_\bY) \\
& = \tr(\bSigma_\bY) - \tr((\bW\bK)^\dagger (\bW\bK) \bM)
\end{align*}
Finally, we notice that the matrix $(\bW\bK)^\dagger (\bW\bK)$ is the orthogonal projection on the space spanned by the rows of $\bW\bK$, which we denote by $\bP_{\bW\bK}$. In particular $\bP_{\bW\bK}$ has the form $\bP_{\bW\bK} = \sum_{i=1}^r \bw_i\bw_i^T$ for some $\bra{\bw_1,\dots,\bw_r}\subset\mathbb{R}^n$ orthonormal vectors and $r\leq p\wedge n$. Therefore, minimize $\ell(\bq(\bW);\bW)$ over $\bW$ it is equivalent to maximize the quantity
\begin{equation*}
\sum_{i=1}^r \bw_i^T \bM \bw_i
\end{equation*}
over the sets of $\bw_1,\dots,\bw_r$ orthonormal vectors of $\mathbb{R}^n$, $r \leq p\wedge n$. Clearly, this is for $\bw_1=\bv_1,\dots,\bw_{p\wedge n}=\bv_{p\wedge n}$. 
This concludes the proof of the lemma.
\end{proof}

\begin{lemma}\label{lemma:hoeffding_hilbert}
Let $X_1,\dots,X_n$ be independent zero-mean r.v.'s taking values in a separable Hilbert space such that $\norm{X_i}\leq c_i$ with probability one and denote
$v = \sum_{i=1}^n c^2_i$. Then, for all $t \geq v$,
it holds 
\begin{equation*}
\mathbb{P}\bra*{ \norm*{ \sum_{i=1}^n X_i } > t } \leq e^{- (t - \sqrt{v})^2 / (2v)}     
\end{equation*}
\end{lemma}
\begin{proof}
The proof can be found in  \citep{boucheron2013concentration}, Example 6.3.
\end{proof}

\begin{lemma}\label{lemma:rkhs_dense}
Be $\mathbb{H}\subset L^2_\bX$ the RKHS defined in the proof of Theorem \ref{theo:energy_gap}. Then $\mathbb{H}$ is dense in $L^2_\bX$.
\end{lemma}
\begin{proof}
First, note that the function $\bx \in\mathbb{R}^n\mapsto k(\bx,\bx)$ is in $L^1(\mathbb{R}^n,d\mu)$. Indeed
\begin{align*}
\int_{\mathbb{R}^n} \int_{\mathbb{S}^n} \psi_\bw(\bx)^2\,d\tau(\bw)\,d\mu(\bx) & = \int_{\mathbb{R}^n} (1+\norm{\bx}^2)\int_{\mathbb{S}^n} \psi_\bw(\bx / \norm{(\bx,1)})^2\,d\tau(\bw)\,d\mu(\bx) 
\\ & \leq (1+\expval\norm{\bX}^2)\max_{\bw,\by\in\mathbb{S}^n} \psi_\bw(\by)^2
\end{align*}
This implies that $\mathbb{H} \subseteq L^2(\mathbb{R}^n,d\mu)$. Now, we would like to show that $V_\sigma$ is dense in $\overline{\mathbb{H}}$, where
$$
V_\sigma = \bra*{\sum_{i=1}^k u_i\psi_{\bw_i} \st \bu \in \mathbb{R}^k, \bw_1,\dots,\bw_k \in\mathbb{S}^n, k \geq 1}
$$
It suffices to show that, for every $\bw \in\mathbb{S}^{n-1}$, there exists a sequence $\bra{f_n}_{n\geq 1} \subset \mathbb{H}$ such that $f_n \to \psi_\bw$ in $L^2_\bX$. Choose $g_k\in L^2(\mathbb{S}^n,d\tau)$ such that $\mathrm{supp}(g_k)\subseteq B_{1/k}(\bw) \doteq \bra{\bv\in\mathbb{S}^n\st \norm{\bv - \bw} \leq 1/k}$, $\int_{\mathbb{S}^n} g(\bv)\,d\tau(\bv) = 1$ and $g_k \geq 0$, and define $f_k \in \mathbb{H}$ as $f_k(x) = \int_{\mathbb{S}^n}g_k(\bv)\psi_\bv(x)\,d\tau(\bx)$. Then
\begin{align*}
\norm{f_k - \psi_\bw}_{L^2(\mathbb{R}^n,d\mu)}^2 & = \int_{\mathbb{R}^n} \parr*{\int_{\mathbb{S}^n} g_k(\bv)(\psi_\bv(\bx) - \psi_\bw(\bx))\,d\tau(\bv)}^2\,d\mu(\bx)
\\ &
\leq  (1+\expval\norm{X}^2)\max_{\substack{\bv \in B_{1/k}(\bw) \\ \by \in \mathbb{S}^n}} (\psi_\bv(\by) - \psi_\bw(\by))^2 \to 0
\end{align*}
as $k\to \infty$. 
This shows that $V_\sigma$ is dense in $\overline{\mathbb{H}}$. Thanks to Theorem 1 in \citep{hornik1991approximation}, it holds that $V_\sigma$ is dense in $L^2(\mathbb{R}^n,d\mu)$. This implies the statement of the lemma.
\end{proof}

\end{document}